\documentclass[preprint,12pt]{elsarticle}

\usepackage{amssymb}
\usepackage{latexsym,amsmath,amsfonts,amsthm,amsxtra,framed,url, graphics,graphicx,epsfig}

\def\pinv{{positively invariant }}
\def\Pinv{{Positively invariant }}
\def\dyns{{dynamical system}}

   \newtheorem{theorem}{Theorem}[section]
   \newtheorem{proposition}[theorem]{Proposition}
   \newtheorem{lemma}[theorem]{Lemma}
   \newtheorem{corollary}[theorem]{Corollary}

   \newtheorem{definition}[theorem]{Definition}
   \newtheorem{example}[theorem]{Example}

   \newtheorem{remark}[theorem]{Remark}

\journal{Applied Mathematics and Computation}

\begin{document}

\begin{frontmatter}



\title{A Novel Unified Approach to Invariance Conditions for a Linear Dynamical System}


\author[siu]{Zolt\'{a}n Horv\'{a}th}
\address[siu]{Department of Mathematics and Computational Sciences, Sz\'{e}chenyi Istv\'{a}n
University, 9026 Gy\H{o}r, Egyetem t\'{e}r 1, Hungary}

\author[lehigh]{Yunfei Song\corref{mycorrespondingauthor}} 
\cortext[mycorrespondingauthor]{Corresponding author}
\ead{yus210@lehigh.edu}

\author[lehigh]{Tam\'{a}s
Terlaky}
\address[lehigh]{Department of Industrial and Systems
Engineering, Lehigh University, \\200 West Packer Avenue, Bethlehem,
PA, 18015, United States}

\begin{abstract}
In this paper, we propose a novel, simple, and unified  approach to explore sufficient and necessary conditions, i.e., invariance conditions,  under which four classic families of convex sets, namely, polyhedra, polyhedral cones, ellipsoids, and
Lorenz cones, are invariant sets for a linear discrete or continuous  dynamical system. 

For discrete dynamical systems, we use the Theorems of Alternatives, i.e., Farkas lemma and \emph{S}-lemma, to obtain simple and general proofs to derive invariance conditions. This novel method establishes a solid connection between optimization theory and dynamical system. Also,  using  the $\textit{S}$-lemma allows us to extend invariance conditions to any set represented by a quadratic inequality. Such sets include nonconvex and unbounded sets. 

For continuous dynamical systems, we use the forward or backward Euler method to obtain the corresponding discrete dynamical  systems  while preserves  invariance. This enables us to develop a novel and elementary method to derive invariance conditions for continuous dynamical systems by using the ones for the corresponding discrete systems. 

Finally, some numerical examples are presented to illustrate these invariance conditions.

\end{abstract}


\end{frontmatter}

\section{Introduction}

\Pinv sets play a key role in the theory and applications of \dyns s. Stability, control and preservation of constraints of \dyns s can be formulated, somehow in a geometrical way, with the help of \pinv sets. For a given \dyns, both of continuous or discrete time, a subset of the state space is called \pinv set for the \dyns\ if containing the system state at a certain time then forward in time all the states remain within the \pinv set. Geometrically, the trajectories cannot escape from a \pinv set if the initial state belongs to the set. The \dyns\ is often a controlled system of which the maximal (or minimal) \pinv set is to be constructed.

It is well known, see e.g., Blanchini \cite{Blanchini2}, Blanchini and Miani \cite{Blanchini3}, and Polanski \cite{polan}, that the
Lyapunov stability theory is  used as a powerful tool in obtaining
many important results in control theory. The basic framework of the
Lyapunov stability theory synthesizes the identification and
computation of a Lyapunov function of a dynamical system. Usually
positive definite quadratic functions serve as  candidate
Lyapunov functions. Sufficient and necessary conditions for positive
invariance of a polyhedral set with respect to discrete dynamical
systems were first proposed by Bitsoris \cite{bits1,bits2}.
 A novel positively invariant polyhedral cone was
constructed by Horv\'{a}th \cite{horva}. The Riccati equation was
proved to be connected with ellipsoidal sets as invariant sets of
linear dynamical systems, see e.g., Lin et al. \cite{lina} and Zhou
et al. \cite{zhou}. 
Birkhoff \cite{birkhoff} proposed a
necessary condition  for positive invariance on a convex cone for linear discrete system. A sufficient and necessary condition for positive invariance on a nontrivial convex set for linear discrete systems was derived by  Elsner \cite{els1}. Stern \cite{stern1} studied the properties of positive invariance on a proper cone for linear continuous systems.
For a more general case, the  mapping from a polyhedral cone to another polyhedral cone was studied by Haynsworth, Fiedler and Pt\'{a}k \cite{tans1}, and the mapping from a convex cone to another convex cone  in finite-dimensional spaces was studied by Tam \cite{tam34, tam12}. Here we note that when the two cones  are the same, then this is equivalent to positive invariance for discrete system. The concept of cross positive matrices, which was introduced
by Schneider and Vidyasagar \cite{schne}, are used as tools  to prove
positive invariance of a Lorenz cone by Loewy and Schneider \cite{loewy}. According to Nagumo's theorem \cite{nagu} and
the theory of cross positive matrices, Stern and Wolkowicz
\cite{stern} presented sufficient and necessary conditions for a
Lorenz cone to be positively invariant with respect to a linear
continuous system. A novel proof of the spectral characterization of
real matrices that leave a polyhedral cone invariant was proposed
by Valcher and Farina \cite{valch}. The spectral properties of the
matrices, e.g., theorems of Perron-Frobenius type, were connected to
set positive invariance by Vandergraft \cite{schne}. Recently, the discrete system has been extended to the case when the state variable belongs to the tangent bundle of a Riemannian manifold or a Lie algebra  by Fiori, see, \cite{fiori2, fiori1}. 
The problem of the unconditional invariance is posed for the first time in the history of control theory by Shipanov \cite{shi1}. 
Gusev and Likhtarnikov \cite{lik1} present a survey of the history of two fundamental results of the mathematical system theory - the Kalman-Popov-Yakubovich lemma and the theorem of losslessness of the \emph{S}-procedure. For an excellent book about the \emph{S}-procedure the reader is referred to \cite{aizer} by Aizerman and Gantmacher. 
An extension of invariance conditions to nonlinear dynamical system can be found in \cite{horv2016}.

 Mathematical modeling of many problems from the real world often leads to differential equations in continuous form. When we solve these differential equations numerically,  we not only need to  obtain a good approximation of the differential equations, but also hope to preserve the basic characteristics of these mathematical variables and models. Invariance preserving is one of the latter type requirements. In fact, there are various characteristics preserving topics, e.g., positivity preserving, strong stability preserving, area preserving, etc, which are extensively studied in recent decades. 
\emph{1). Positivity Preserving:} 
Positivity preserving is an important topic in the numerical analysis community, see, e.g.,  \cite{horva, horv05, zhang1, zhang2, zhang3}. Positivity preserving is equivalent to invariance preserving in the positive orthant, i.e.,  consider the positive orthant, which is a polyhedral cone. Let us assume that the positive orthant  is an invariant set for a continuous system, and assume that it  is also an invariant set for the discrete system which is obtained by using a discretization method with a certain steplength. In practice, many variables, e.g.,  energy, density, mass, etc, are nonnegative. When these variables are used in some mathematical models in a continuous form, e.g., in the heat equation, one should choose appropriate discretization method with appropriate steplength such that solution of the the discretized systems are also nonnegative.    
\emph{2). Strong Stability Preserving (SSP):}
Strong stability preserving (SSP) numerical methods are developed to solve ordinary differential equations, see, e.g., \cite{sspbook, siga2},  etc.  Particularly, SSP numerical method are used for the time integration of semi-discretizations of hyperbolic conservation laws.  It is well known that the exact solutions of scalar conservation laws holds the property that  total variation does not increase in time, see, e.g., \cite{siga2}. 
SSP methods are also referred to as total variation diminishing methods. These   are higher order numerical methods that also preserve this property.  
\emph{3). Area Preserving-Symplectic Methods:}
Intuitively, a map from the phase-plane to itself is said to be symplectic if it preserves areas.
 In mathematics, a  matrix $M\in \mathbb{R}^{2n\times 2n}$ is called symplectic if it satisfies the condition
$M^T \Omega M = \Omega,$
where 
$
\Omega =
\begin{bmatrix}
0 & I_n \\
-I_n & 0 \\
\end{bmatrix}.$ A symplectic map is a real-linear map $T$ that preserves a symplectic form $f$, i.e., 
$ f(Tx,Ty)=f(x,y) $
for all $x, y$, see, e.g., \cite{meyer1}.
A numerical one-step method $x_{n+1} = D_{\Delta t}(x_n)$ is called symplectic if, when applied
to a Hamiltonian system, the discrete flow $x\rightarrow D_{\Delta t}(x)$ is a symplectic map
for all sufficiently small step sizes, see, e.g., \cite{feng1, mark1}, etc. There is one compelling example that shows symplectic methods are the right way to solve planetary trajectories. If we solve the trajectory of the earth using forward Euler method, then the discrete trajectory will spiral away from the sun. If we use backward Euler method, then the discrete trajectory will sink into the sun. If we use symplectic methods, then the discrete trajectory will stay on the original continuous trajectory. 

In many applications, the models are represented as a partial differential equation (PDE),  e.g., heat equation, then certain numerical methods, e.g., finite difference methods, finite element methods, etc., may be first applied to the spatial variable to obtain a ODE (dynamical system). The numerical methods for ODE are then used to obtain the discrete form of the model. Therefore,  invariance condition for a ODE (dynamical system) is crucial for models even within a PDE form. We point out that the invariance condition for the numerical methods for the spatial variable of the PDE is an important research topic but out of the scope of this paper.

In this paper we deal with dynamical systems in finite dimensional spaces and introduce a novel and unified method for the determination of whether a set  is a \pinv set for a linear \dyns. Here the sets are ellipsoids, polyhedral sets or - not necessarily convex - second order sets including Lorenz cones. In addition, we formulate optimization methods to check the resulting equivalent conditions. 

The main tool in the continuous time case consists of the explicit computation of the tangent cones of the \pinv sets and their application along the lines of the Nagumo theorem \cite{nagu}. This theorem says that a set is positively invariant, under some conditions on solvability of the underlying differential equation, if and only if at each point of the set, the vector field of the differential equation  points toward the tangent cone at that point. The resulting conditions are constructive in the sense that they can be checked by well established optimization methods.
Our unified approach is based on optimization methodology. The analysis in the discrete case is based on the theorems of alternatives of optimization, namely on the Farkas lemma \cite{roos} and the \textit{S}-lemma \cite{polik,yaku}. 
The name S-lemma is due to the name of a Lagrange function that corresponds to the constrained optimization in \cite{lik1}. Lagrange multipliers method as a penalty method of constrained nonlinear optimization can refer to \cite{flec1}.
Let us mention that the technique with the tangent cones in the continuous time case and the theorem of alternatives of optimization in the discrete case show common features. 

First, in the paper, we consider various sets as candidates for
positively invariant sets with respect to a discrete system.
Sufficient and necessary conditions for the four types of sets are
derived using the Farkas lemma \cite{roos} and the \textit{S}-lemma
\cite{polik,yaku}, respectively. The Farkas lemma and the
\textit{S}-lemma are frequently referred to as Theorems of the
Alternatives in the optimization literature. Note that the approach
based on the Farkas lemma is originally due to Hennet \cite{hennet}.  Our approach, based on the \textit{S}-lemma for
ellipsoids and Lorenz cones, is not only simpler compared to the
traditional Lyapunov theory based approach, but also highlights the strong
relationship between control and optimization theories. It also enables us to extend  invariance conditions to any set represented by a quadratic inequality. Such sets include nonconvex and unbounded sets.
Positively
invariant sets for continuous systems are linked to the ones for
discrete systems by applying Euler method. The forward Euler method
or backward Euler method is used to discretize a continuous system
to a discrete system. According to \cite{blannew, blan12, song2}, we have that both
the continuous and discrete systems can share the same set as a
positively invariant set when forward or backward Euler methods are used and when the discretization steplength is bounded by a certain
value. In \cite{song2}, we prove that there exists a uniform upper bound of the
steplength for both the forward and backward Euler methods such that the discrete and continuous systems
can share a polyhedron or a polyhedral cone as a positively invariant set (for an ellipsoid or a
Lorenz cone, there exists a uniform upper bound of the steplength for the
backward Euler method). An efficient algorithm to derive the uniform steplength threshold for invariance preserving for certain discretization methods on a polyhedron  is presented in  \cite{song3}.
 Then, sufficient and necessary
conditions under which the four types of convex sets are positively
invariant sets for the continuous systems are derived by using 
Euler methods and the corresponding sufficient and necessary
conditions for  the discrete systems.

The main novelty of this paper is that we propose a simple, novel,
unified approach to derive invariance conditions for the four types of sets to be
positively invariant sets with respect to discrete systems.  Our
approach is based on the so-called Theorems of Alternatives, i.e., Farkas lemma and \emph{S}-lemma. For discrete systems, the Farkas lemma is used for polyhedral sets, while the \textit{S}-lemma is used for  ellipsoids and Lorenz
cones. We also establish a
framework according to Euler methods to derive invariance conditions for the
four types of sets with respect to the continuous systems to be
positively invariant. Although some theorems presented in this paper are known, there is no existing paper considering  invariance conditions for the four types of sets, and both for discrete and continuous dynamical systems together in a unified framework. We also strengthen the power of Euler methods as a tool to study invariance conditions to build  connection between continuous and discrete  dynamical systems. 

\emph{Notation and Conventions.} To avoid unnecessary repetitions,
the following notations and conventions are used in this paper.
A dynamical system, positively
invariant, and sufficient and necessary condition for positive
invariance are called a \emph{system},  \emph{invariant}, and
\emph{invariance condition}, respectively.
  The sets considered in this paper are
non-empty, closed, and convex sets if not specified otherwise.
    The interior and the boundary of a set $\mathcal{S}$ is denoted by
int$(\mathcal{S})$ and  $\partial \mathcal{S}$, respectively.
  A symmetric positive definite, positive semidefinite, negative definite, or negative semidefinite matrix $Q$ is denoted by $Q\succ0$, $ Q\succeq0, Q\prec0,$  or $
Q\preceq0$, respectively.
The $i$-th row of a matrix $G$ is denoted by $G_i^T.$
   The
eigenvalues of a real symmetric matrix $Q$, whose eigenvalues are
always real, are ordered as 
$\lambda_1\geq\lambda_2\geq...\geq\lambda_n,$ and the corresponding orthonormal set of  
eigenvectors is denoted by $\{u_1,u_2,...,u_n\}$.
  The
  spectral radius  of $Q$ is represented by $\lambda(Q)=\max\{|\lambda_i(Q)|\}$, and inertia$\{Q\}=\{\alpha,\beta,\gamma\}$ indicates that the number
of positive, zero, and negative eigenvalues of $Q$ are $\alpha,\beta$, and $\gamma$, respectively.
  The index set $\{1,2,...,n\}$ is
denoted by $\mathcal{I}(n).$
 The inner product of vectors $x,y\in \mathbb{R}^n$ is represented by
  $x^Ty$.

This paper is organized as follows: in Section \ref{sec2}, the related basic concepts and theorems are introduced. Our main results are shown in Section \ref{sec:invcond}, in which invariance conditions of polyhedral sets, ellipsoids, and Lorenz cones for continuous and discrete systems are presented. In Section \ref{sec:exam}, some numerical examples are given to illustrate the invariance conditions presented in Section \ref{sec:invcond}. Finally, our conclusions are summarized in Section \ref{sec:con}.

\section{Basic Concepts and Theorems}\label{sec2}
In this section, the basic concepts and theorems related to invariant sets for dynamical systems are introduced. 

\subsection{Linear Dynamical System}
In this paper, we consider discrete and continuous linear dynamical
systems, respectively described by the following equations:
\begin{equation}\label{dyna2}
x_{k+1}=B_kx_k,
\end{equation}
\begin{equation}\label{dyna1}
\dot{x}(t)=Ax(t),
\end{equation}
where $B_k, A\in \mathbb{R}^{n\times n}$ are constant real matrice, $
x_k,  x(t)\in \mathbb{R}^n$ are the state variables, $t\in
\mathbb{R}$, and $k\in \mathbb{N}$. We may assume, without loss of
generality, that $B_k $ and $A$ are not the zero matrix. The study of invariant sets is
the main subject of this paper, thus now we introduce invariant sets
for both discrete and continuous linear systems.
Note that equations (\ref{dyna2}) and (\ref{dyna1}) can be treated as autonomous systems or as controlled systems. In the latter case, the coefficient matrix $B_k$ and $A$ in  (\ref{dyna2}) or (\ref{dyna1}) can be represented in the form of $C+DF$, where $C$ is the open-loop state matrix, $D$ is the control matrix, and $F$ is the gain matrix\footnote{For simplicity, we take a discrete system as an example. In this case, the system is represented as follows: $x_{k+1}=Cx_k+Du_k$, where $x_k$ is the state variable, $u_k$ is the control variable, and $u_k=Fx_k$. Thus,  this equation is equivalent to $x_{k+1}=(C+DF)x_k.$}. 

\begin{definition}\label{def1}
A set $\mathcal{S}\subseteq \mathbb{R}^n$ is an  invariant  set for
the discrete system (\ref{dyna2}) if  $ x_k\in \mathcal{S}$
implies $ x_{k+1}\in \mathcal{S}$, for all
 $k\in \mathbb{N}$.
\end{definition}

\begin{definition}\label{def2}
A set $\mathcal{S}\subseteq \mathbb{R}^n$ is an invariant set for
the continuous system (\ref{dyna1}) if $x(0)\in \mathcal{S}$ implies
$x(t)\in\mathcal{S}$,  for all $t\geq0$.
\end{definition}

In fact, the sets given in Definition \ref{def1}
and \ref{def2}  are conventionally  referred to as positively invariant sets. Considering that only positively invariant sets are studied in this
paper, we simply call them
invariant sets. One can prove the following
properties:  the operators $B_k$ (or\footnote{The exponential function with respect to a matrix is defined as $e^{At}=\sum_{k=0}^\infty \frac{1}{k!}(A^kt^k)$.} for all $t\geq0,$ $e^{At}$) leave $\mathcal{S}$ invariant if $\mathcal{S}$ is an invariant set for the discrete (or continuous) systems.

\begin{proposition}\label{prop1} \emph{\cite{bellm, boyd}}
The set $\mathcal{S}$ is an invariant set for the discrete system
(\ref{dyna2}) if and only if  $B_k\mathcal{S}\subseteq \mathcal{S}$.
Similarly, the set $S$ is an invariant set for the continuous system
(\ref{dyna1}) if and only if for all $t\geq0$,
$e^{At}\mathcal{S}\subseteq \mathcal{S}$.
\end{proposition}

\subsection{Convex Sets}\label{conset}
In this paper, we investigate invariance conditions for 
some classical convex sets, namely polyhedral sets, ellipsoids,  and
Lorenz cones.

A \emph{polyhedron}, denoted by $\mathcal{P}\subseteq \mathbb{R}^n$, can be defined as the
intersection of a finite number of half-spaces:
\begin{equation}\label{poly1}
\mathcal{P}=\{x\in \mathbb{R}^n\,|\,Gx\leq b\},
\end{equation}
where $G\in \mathbb{R}^{m\times n} $, $ b\in \mathbb{R}^m$, or
equivalently, as the sum of the convex combination of a finite number of points
and the conic combination of a finite number of  vectors:
\begin{equation}\label{poly2}
\mathcal{P}=\Big\{x\in \mathbb{R}^n\,|\,x=\sum_{i=1}^{\ell_1}\theta_i
x^i+\sum_{j=1}^{\ell_2}\hat{\theta}_j\hat{x}^j,~
\sum_{i=1}^{\ell_1}\theta_i=1, \theta_i\geq0, \hat{\theta}_j\geq
0\Big\},
\end{equation}
where $x^1,...,x^{\ell_1},\hat{x}^1,...,\hat{x}^{\ell_2}\in
\mathbb{R}^n$.  The \emph{vertices} of $\mathcal{P}$ form a subset of $x^i, i\in \mathcal{I}(\ell_1)$, and the \emph{extreme rays} of $\mathcal{P}$ are represented as $x^i+\alpha\hat{x}^j, \alpha>0,$ for some $i\in \mathcal{I}(\ell_1)$ and $j\in \mathcal{I}(\ell_2).$  We highlight that a bounded polyhedron, i.e., $\ell_2=0$ in
(\ref{poly2}), is called a \emph{polytope}.

A \emph{polyhedral cone}, denoted by $\mathcal{C_P}\subseteq \mathbb{R}^n$, can be also considered
as a special class of polyhedra, and  it can be defined as:
\begin{equation}\label{polycone1}
\mathcal{C_P}=\{x\in \mathbb{R}^n\,|\,Gx\leq 0\},
\end{equation}
or equivalently,
\begin{equation}\label{polycone2}
\mathcal{C_P}=\Big\{x\in
\mathbb{R}^n\,|\,x=\sum_{j=1}^{\ell}\hat{\theta}_j\hat{x}^j,~
\hat{\theta}_j\geq 0\Big\},
\end{equation}
where $G\in \mathbb{R}^{m\times n}$, and
$\hat{x}^1,...,\hat{x}^{\ell}\in \mathbb{R}^n$.

An \emph{ellipsoid}, denoted by $\mathcal{E}\subseteq \mathbb{R}^n$, centered at the
origin, is defined as:
\begin{equation}\label{elli}
\mathcal{E}=\{x\in\mathbb{ R}^n \,|\, x^TQx\leq 1\},
\end{equation}
where  $Q\in \mathbb{R}^{n\times n}$ and $Q\succ0$. Any ellipsoid with nonzero center can be transformed to an ellipsoid centered at the origin.

A \emph{Lorenz cone}\footnote{A Lorenz cone is sometimes also called an
ice cream cone,  a second order cone, or an ellipsoidal cone.}, denoted by $\mathcal{C_L}\subseteq \mathbb{R}^n$,
with vertex at the origin,  is defined as:
\begin{equation}\label{ellicone}
\mathcal{C_L}=\{x\in \mathbb{R}^n\,|\,x^TQx\leq 0,~ x^Tu_n\geq0\},
\end{equation}
where $Q\in \mathbb{R}^{n\times n}$ is a symmetric nonsingular
matrix with one negative eigenvalue $\lambda_n$, i.e., ${\rm
inertia}\{Q\}=\{n-1,0,1\}$, and $u_n$ is the eigenvector corresponding to the only negative eigenvalue $\lambda_n$. Similar to ellipsoids, any Lorenz cone with nonzero vertex can be transformed to a Lorenz cone with vertex at the origin.
For every Lorenz cone given as in (\ref{ellicone}), there exists an orthonormal basis
$\{u_1,u_2,...,u_n\}$, {i.e.}, $u_i^Tu_j=\delta_{ij}$, where $u_i$
is the eigenvector corresponding to the eigenvalue, $\lambda_i$, of $Q$, and
$\delta_{ij}$ is the Kronecker delta function, such that
$Q=U\Lambda^{\frac{1}{2}}\tilde{I}\Lambda^{\frac{1}{2}}U^T,$ 
where
$\Lambda^{\frac{1}{2}}=\text{{diag}}\{\sqrt{\lambda_1},...,\sqrt{\lambda_{n-1}},\sqrt{-\lambda_n}\}$
and $\tilde{I}=\text{{diag}}\{1,...,1,-1\}$. In particular,
the Lorenz cone with $Q=\tilde{I}$ is denoted by $\mathcal{K}_n$, then we have $ \mathcal{K}_n=\{x\in \mathbb{R}^n\,|\, x^T\tilde{I}x\leq 0,
x^Te_n\geq0\}, $ where $e_n=(0,...,0,1)^T.$ We call $\mathcal{K}_n$ the  \emph{standard Lorenz cone}.


\subsection{Basic Theorems}

The  Farkas lemma \cite{roos} and the \textit{S}-lemma \cite{polik,yaku},
both of which are also called the Theorem of Alternatives,
are fundamental tools to derive  invariance conditions for
discrete systems in our study. The \textit{S}-lemma proved by Yakubovich
\cite{yaku} is somewhat analogous to a special case of the
nonlinear Farkas lemma, see P\'{o}lik and Terlaky \cite{polik}.

\begin{theorem}\label{farkas2} \textbf{\emph{(Farkas lemma \cite{roos})}}
Let $P\in \mathbb{R}^{m\times n}, $ $d\in \mathbb{R}^m, c\in\mathbb{R}^n,$ and $\beta\in\mathbb{R}$. Then
the following two statements are equivalent:
\begin{enumerate}
  \item There is no $y\in \mathbb{R}^m$, such that  $P^Ty\leq c$ and  $d^Ty> \beta$;
    \item  There exists a vector $z\in \mathbb{R}^n$, such that $z\geq0,Pz=d,$
  and $c^Tz\leq \beta$.
\end{enumerate}
\end{theorem}

\begin{theorem}\label{slemm2}
\emph{\textbf{(\textit{S}-lemma \cite{polik,yaku})}}  Let
$g(y),r(y):\mathbb{R}^n\rightarrow \mathbb{R}$ be quadratic functions,
and suppose that there is a $\hat{y}\in \mathbb{R}^n$ such that
$r(\hat{y})<0$. Then the following two statements are equivalent:
\begin{enumerate}
  \item There exists no $y\in \mathbb{R}^n$, such that $g(y)<0, r(y)\leq0.$
  \item There exists a scalar $\rho\geq 0$, such that
  $g(y)+\rho r(y)\geq0,$ for all $y\in \mathbb{R}^n.$
\end{enumerate}
\end{theorem}

Proposition \ref{prop1} allows us to use the Theorems of Alternatives \ref{farkas2} and \ref{slemm2} to derive  invariance conditions for discrete systems.  According to Proposition \ref{prop1},  to prove that a set $\mathcal{S}$ is an invariant set for a discrete system, we need to prove $A\mathcal{S}\subseteq \mathcal{S}$, which is equivalent to $(\mathbb{R}^n\setminus\mathcal{S})\cap(A\mathcal{S})=\emptyset.$ Since we assume that $\mathcal{S}$ is a closed set, we have that $\mathbb{R}^n\setminus\mathcal{S}$ is an open set.  Open sets are usually represented by strict inequalities. As the Theorems of Alternatives include strict inequalities, they provide the proper tools to characterize invariance conditions for continuous and discrete systems. 
This is one of the statements in the Theorems of Alternatives \ref{farkas2} or \ref{slemm2}.\\

For  invariance conditions for continuous systems, the concept of  \emph{tangent cone} plays an important role in our analysis.

\begin{definition}\label{intro:def-TanCone}
Let
$\mathcal{S}\subseteq\mathbb{R}^n$ be a closed convex set, and $x\in \mathcal{S}$.
 The tangent cone  of
$\mathcal{S}$ at $x$, denoted by $\mathcal{T}_\mathcal{S}(x)$, is given as
\begin{equation}\label{tancon}
\mathcal{T}_\mathcal{S}(x)=\Big\{y\in
\mathbb{R}^n\;\Big|\;\underset{t\rightarrow0^+}{\lim\inf}\frac{{\emph{dist}}(x+ty,\mathcal{S})}{t}=0\Big\},
\end{equation}
where   $\emph{dist}(x,\mathcal{S})=\inf_{s\in\mathcal{S}}\|x-s\|.$
\end{definition}

A geometrical interpretation of tangent cones is given by the left side picture of  Figure \ref{fig1}. The tangent cone at vertex $c_1$ is the red colored NW-SE shaded cone, and the tangent cone at extreme point $c_2$ is the green color SW-NE shade half space, which is also a cone. 

\begin{figure}[h]
    \centering
     \includegraphics[width=0.5\textwidth]{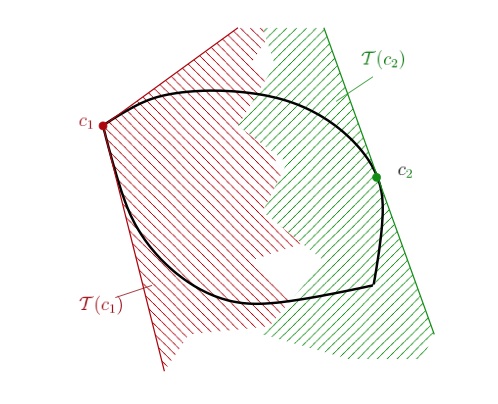}
    \includegraphics[width=0.4\textwidth]{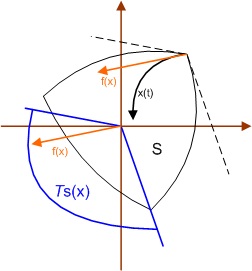}
    \caption{Tangent Cone (left) and Nagumo Theorem (right).}
    \label{fig1}
\end{figure}

The following classic result proposed by Nagumo \cite{nagu}
provides a general criterion to determine whether a  closed convex set is an invariant set for a continuous
system. This theorem, however, is not valid for discrete systems, for which
one can find a counterexample in \cite{Blanchini}.

\begin{theorem}\label{nagumo}
\emph{\textbf{(Nagumo  \cite{Blanchini, nagu})}} Let
$\mathcal{S}\subseteq\mathbb{R}^n$ be a closed convex set, and
assume that the system $\dot{x}(t)=f(x(t))$, where
$f:\mathbb{R}^n\rightarrow \mathbb{R}^m$ is a continuous mapping,
admits a globally unique solution for every initial point $x(0)\in
\mathcal{S}$. Then $\mathcal{S}$ is an invariant set for this system
if and only if
\begin{equation}\label{cond3}
f(x)\in \mathcal{T}_\mathcal{S}(x),\text{  for all }  x\in \partial
\mathcal{S},
\end{equation}
where  $\mathcal{T}_\mathcal{S}(x)$ is the tangent cone of
$\mathcal{S}$ at $x.$
\end{theorem}

Nagumo's Theorem \ref{nagumo} has an intuitive geometrical interpretation as follows: for any trajectory that starts in $\mathcal{S}$,
it has to go through $\partial\mathcal{S}$ if it  goes out of $\mathcal{S}$. Then one needs only to consider the property of this trajectory on
$\partial\mathcal{S}$. Note that $f(x)$ is the derivative of the
trajectory, thus (\ref{cond3}) ensures that the trajectory will
point inside $\mathcal{S}$ on the boundary, which means $\mathcal{S}$ is an invariant set. The
disadvantage of Theorem \ref{nagumo}, however, is that it may be
difficult to verify whether (\ref{cond3}) holds for  all
points on the boundary of a given set.
According to Nagumo's Theorem \ref{nagumo}, the key is to derive the formula of the tangent cone on the boundary of the set. An intuitive interpretation is given in the right side subfigure of Figure \ref{fig1}.

We use Euler methods to discretize  continuous system (\ref{dyna1}) to derive a discrete system, because  for sufficiently small step size they  preserve the invariance of a set, i.e., a set, which is an invariant set for a continuous system, is also an invariant set for the corresponding derived discrete system. Here we formally present these results as follows. The first statement can be found in \cite{blan6, Blanchini, blannew, song2}, and the second statement can be found in \cite{song2}.
\begin{theorem}\label{songthm}
Assume a polyhedron $\mathcal{P}$, polyhedral cone $\mathcal{C_P}$, ellipsoid $\mathcal{E}$ or Lorenz cone $\mathcal{C_L}$  is an invariant set for the continuous system (\ref{dyna1}). Then 
\begin{itemize}
\item there exists a $\hat\tau>0$, such that  $\mathcal{P}$ (or $\mathcal{C_P}$)  is also an invariant set for the discrete system $x_{k+1}=(I+A\Delta t)x_k$ for all $0\leq \Delta t\leq \hat\tau$, and 
\item there exists a $\tilde\tau>0$, such that $\mathcal{P}$ ($\mathcal{C_P}, \mathcal{E}$ or $\mathcal{C_L}$) is also an invariant set for the discrete system $x_{k+1}=(I-A\Delta t)^{-1}x_k$ for all $0\leq \Delta t\leq \tilde\tau$.
\end{itemize}
\end{theorem}
\begin{remark}
The first statement in Theorem \ref{songthm} means that the forward Euler method preserves the invariance of polyhedral set, while the second statement means that the backward Euler method preserves the invariance of polyhedral set, ellipsoid, and Lorenz cone. 
\end{remark}


\section{Invariance Conditions}\label{sec:invcond}
In this section, we present the invariance conditions, i.e., sufficient and necessary conditions under which  polyhedral sets, ellipsoids, and Lorenz cones are  invariant sets for discrete and continuous systems. For each convex set, the invariance conditions for discrete systems are first derived by  using the Theorems of Alternatives, i.e., the Farkas lemma or the $S$-lemma. Then the invariance conditions for continuous systems are derived by using a discretization method to discretize the continuous system and applying the invariance conditions for the obtained discrete systems.

\subsection{Polyhedral Sets}
Since every polyhedral set has two different representations  as shown in Section \ref{conset}, we present the invariance conditions for both forms, respectively.  \emph{Nonnegative} and \emph{essentially nonnegative matrices} are used in the invariance conditions. 
\begin{definition}
 A matrix $H$ is called a nonnegative  matrix, denoted by $H\geq0$, if $H_{ij}\geq0$ for all $i,j$. A matrix $L$ is called an essentially nonnegative matrix\footnote{An essentially nonnegative matrix, see e.g., \cite{caste}, is also called Metzler matrix or quasipositive matrix, see, e.g., \cite{berman}. },
denoted by $L{\geq}_o 0$,  if $L_{ij}\geq0$ for $i\neq j$.
\end{definition}

\subsubsection{Invariance Conditions for Discrete Systems}
The invariance condition of a polyhedral sets given as in (\ref{poly1}) for
a discrete system is presented in Theorem \ref{polythm}. The study of invariance condition of polyhedral sets for discrete system can be traced back to Bitsoris in \cite{bits1, bits2}, which consider a special class of polyhedral sets that is symmetric with respect to the origin.  We give a more straightforward proof here by using the Farkas lemma for the polyhedral set in the form of (\ref{poly1}). It was brought to our attention recently that the result is the same as the one presented by Hennet \cite{hennet}, which also uses the Farkas lemma. To keep the integration of the paper, we also present the proof explicitly. 

\begin{theorem}\label{polythm}
\textbf{\emph{(Hennet \cite{hennet})}} A polyhedron $\mathcal{P}$ given as in (\ref{poly1}) is an invariant set for the discrete
system (\ref{dyna2}) if and only if \footnote{The referee proposes an easy way to show the ``if" part: let $x\in \mathcal{P},$ i.e., $Gx\leq b$. Since $H\geq0, HG=GB_k$ and $Hb\leq b$, we have $GB_kx=HGx\leq Hb\leq b$, i.e., $B_kx\in \mathcal{P}$.} there exists a matrix $H\in \mathbb{R}^{m\times m}$,
such that $H\geq0, HG=GB_k$ and $ Hb\leq b.$
\end{theorem}

\begin{proof}
We have that $\mathcal{P}$ is an invariant set for the discrete system (\ref{dyna2}) if and only if $B_k\mathcal{P}\subseteq\mathcal{P},$ which is the same as $\mathcal{P}\subseteq\mathcal{P}'=\{x\;|\;GB_kx\leq b\}.$ Note that $\mathcal{P}\subseteq\mathcal{P}'$ if and only if for every $i\in \mathcal{I}(m)$, we have
$$\{x\,|\,Gx\leq b\}\cap \{x\,|\,(GB_k)_i^Tx>b_i\}=\emptyset,$$
i.e., the inequality system $Gx\leq b$ and $(GB_k)_i^Tx>b_i$ has no solution.  According to the Farkas lemma \ref{farkas2}, this is equivalent to that there exists a vector $h_i\geq0$, such
that $G^Th_i=(GB_k)_i,$ and $b^Th_i\leq b_i. $ We let $H=[h_1,h_2,...,h_m],$ then we have $H\geq0, HG=GB_k$ and $Hb\leq  b$. The proof is complete. 
\end{proof}

We highlight that Castelan and
Hennet \cite{caste} present an algebraic characterization of  the
matrix $G$ satisfying the conditions in  Theorem \ref{polythm}. They prove that given $B_k$ and $G$, there exists a matrix $H$ satisfying $HG=GB_k$ if and only if  the
kernel of $G$ is an $B_k$-invariant subspace.

The invariance condition of a polyhedral set given as in (\ref{poly2}) for discrete systems is provided in Theorem
\ref{polythm2}. Note that a similar result is presented in \cite{Blanchini}, which
considers only the case when the set is a polytope. Invariance condition of  a polytope is presented in \cite{Blanchini}, while invariance condition of a polyhedral cone is presented in \cite{tiwa}. Here we integrate these two results in one theorem. 

\begin{theorem}\label{polythm2}
A polyhedron $\mathcal{P}$ given as in (\ref{poly2}) is an
invariant set for the discrete system (\ref{dyna2}) if and only if
there exists a  matrix $L\in
\mathbb{R}^{(\ell_1+\ell_2)\times (\ell_1+\ell_2)}$, such that $L\geq0, XL=B_kX$ and $ \bar{1}^TL=
\bar{1}^T,$ where
$X=[x^1,...,x^{\ell_1},\hat{x}^1,...,\hat{x}^{\ell_2}]$,
$\bar{1}^T=(1_{\ell_1}^T,0_{\ell_2}^T)$.
\end{theorem}

\begin{proof} 
Note that $\mathcal{P}$ given as in (\ref{poly2}) is an
invariant set for the discrete system if and only if $B_kx^i\in
\mathcal{P}, $ for all $i\in
\mathcal{I}(\ell_1)$, and $B_k(O^+\mathcal{P})\subseteq O^+\mathcal{P}$, where $O^+\mathcal{P}$ denotes the recession cone of $\mathcal{P}$. Clearly, $B_kx^i\in
\mathcal{P} $ for all $i\in
\mathcal{I}(\ell_1)$ is equivalent to that there exist ${\theta}_{p_1}^{i}, \hat{\theta}_{p_2}^{i}\geq0$, $p_1\in\mathcal{I}(\ell_1), p_2\in \mathcal{I}(\ell_2),$ with $\sum_{p_1=1}^{\ell_1}\theta_{p_1}^{i}=1$, such that 
$B_kx^i=\sum_{p_1=1}^{\ell_1}\theta_{p_1}^{i}x^{p_1}+\sum_{p_2=1}^{\ell_2}\hat{\theta}_{p_2}^{i}\hat{x}^{p_2}.$
Since  $O^+\mathcal{P}$ is generated by  $\hat x^j$, where $j\in \mathcal{I}(\ell_2)$, we have that $B_k(O^+\mathcal{P})\subseteq O^+\mathcal{P}$ can be rewritten as  
$ B_k\hat{x}^j\in O^+\mathcal{P},$ for all  $j\in \mathcal{I}(\ell_2)$. Then $B_k(O^+\mathcal{P})\subseteq O^+\mathcal{P}$ is equivalent to that there exist 
 ${\theta}_{p_2}^{j}\geq0, p_2\in\mathcal{I}(\ell_2), $  such that 
$B_k\hat{x}^j=\sum_{p_2=1}^{\ell_2}\hat{\theta}_{p_2}^{j}\hat{x}^{p_2}.$
Let  $L=[\theta^{1},..,.\theta^{\ell_1},
\hat\theta^{1},...,\hat\theta^{\ell_2}]$, then the theorem is immediate. 
\end{proof}

A polyhedral cone is a special polyhedral set, thus
we have the following  invariance condition of a polyhedral cone for discrete
systems.

\begin{corollary}\label{polythmcon4}
1). A polyhedral cone $\mathcal{C_P}$ given as in (\ref{polycone1})
is an invariant set for the discrete  system (\ref{dyna2}) if
and only if there exists a  matrix $H\in \mathbb{R}^{m\times m}$,  such that $H\geq0$ and $HG=GB_k$. 

2). A
polyhedral cone $\mathcal{C_P}$ given as in (\ref{polycone2}) is an invariant
set for the discrete system (\ref{dyna2}) if and only if there
exists a matrix $L\in
\mathbb{R}^{\ell\times \ell}$, such that $L\geq0$ and $XL=B_kX$, where
$X=[\hat{x}^1,...,\hat{x}^{\ell}]$.
\end{corollary}

For a given polyhedral set and
a discrete system, according to Theorem
\ref{polythm} (Theorem \ref{polythm2}, or Corollary
\ref{polythmcon4}), to determine  whether the set is an invariant set for
the system is equivalent to verify the existence of a nonnegative
matrix $H$ (or $L$), which is actually a linear optimization
problem. Rather than computing $H$ (or $L$) directly, it is more
efficient to sequentially solve some small subproblems. Let
us choose polyhedron $\mathcal{P}$ as given in (\ref{poly1}) and
Theorem \ref{polythm} as an example to illustrate this idea. We can
sequentially examine the feasibility of the subproblems. Find $h_i\in \mathbb{R}^n$, such that 
$h_i^TG=G_i^TB_k, \, h_i\geq0,$ and $ h_i^Tb\leq b_i$, for all $i\in\mathcal{I}(n)$. Clearly, these are linear feasibility problems which can be considered as a special case of  linear optimization problems, see, e.g., \cite{hill}. A linear optimization problem can be solved in polynomial time, e.g., by using interior point methods \cite{roos}. If all of these linear optimization problems  are
feasible, then their solutions forms such a nonnegative matrix $H$.
Otherwise, we can conclude that the set is not an invariant set for the system, and  the
computation is terminated at the first infeasible subproblem.

\subsubsection{Invariance Conditions for Continuous Systems}

According to \cite{song2}, we have that both the forward and backward Euler methods are invariance preserving for a polyhedral set. Blanchini
\cite{blan6,Blanchini} presents  the connection between invariant sets for continuous and discrete systems by using the forward Euler method. The discrete system obtained by using the forward Euler method is refereed to as Euler Approximating System \cite{blan6,Blanchini}. 
We first present the following invariance condition which is obtained by using Nagumo's Theorem \ref{nagumo}. For $x\in \mathcal{P},$ let $\mathcal{I}_x$ denote the set of indices of the constraints which are active at $x$, i.e., the corresponding linear inequality holds as equality at $x$. Clearly, we have $x\in \partial \mathcal{P}$ if and only if $\mathcal{I}_x\neq\emptyset.$

\begin{lemma}\label{lemmap1}
Let a polyhedron $\mathcal{P}$ be given as in (\ref{poly1}), and $\mathcal{I}_x\neq \emptyset$ for all $x\in \mathcal{P}$. Then $\mathcal{P}$ is an invariant set for the continuous system (\ref{dyna1}) if and
only if for every $x\in \partial{\mathcal{P}}$, i.e.,  $G_{i}^Tx= b_{i}$,  we have
\begin{equation}
G_{i}^TAx\leq 0, ~~~\text{ for all } i\in\mathcal{I}_x.
\end{equation}
\end{lemma}
\begin{proof}
For all $x\in \partial \mathcal{P}$, the tangent cone at $x$ is 
$\mathcal{T}_{\mathcal{P}}(x)=\{y\,|\,G_{i}^Ty\leq 0, i\in\mathcal{I}_x\}$ for all $i\in \mathcal{I}_x$ (see  \cite[p.138]{hul}). Then the lemma immediately follows from
Nagumo's Theorem \ref{nagumo}.
\end{proof}

We now present another invariance condition of a polyhedron  in the form of (\ref{poly1}) for the continuous system (\ref{dyna1}). The following theorem also refers to Castelan and Hennet \cite[Proposition 1]{caste}.

\begin{theorem}\label{polythmcon} 
A polyhedron $\mathcal{P}$ given as in (\ref{poly1}) is an
invariant set for the continuous  system (\ref{dyna1}) if and only
if there exists a  matrix $\tilde{H}\in \mathbb{R}^{m\times m}$, such that
$\tilde{H}\geq_o0, \tilde{H}G=GA$ and $ \tilde{H}b\leq 0.$
\end{theorem}

\begin{proof}  
We first consider the ``if" part. 
Noting that  $\tilde{H}G=GA$, we have $\tilde{H}_{i}^TGx=G_{i}^TAx,$ for every $i\in\mathcal{I}(n)$. 
Since $\tilde{H}{\geq}_o0$ and
$x\in \partial\mathcal{P}$, 
\begin{equation}\label{revis1}
\begin{array}{rl}
   \text{when }  j=i , &\text{we have } \tilde{h}_{ii}\in \mathbb{R}  \text{ and } G_i^Tx=b_i,\\
    \text{when } j\neq i, & \text{we have }\tilde{h}_{ij}\geq0\text{ and }  G_j^Tx\leq b_j,
\end{array}
\end{equation}
where $\tilde{h}_{ij}$ is the $(i,j)$-th entry of $\tilde{H}.$
According to (\ref{revis1}), we have 
$\sum_{j=1}^m\tilde{h}_{ij}(G_j^Tx-b_j)\leq0,$ i.e., $\tilde{H}_{i}^TGx\leq\tilde{H}_{i}^Tb$. Since $\tilde{H}b\leq0,$ we have
$\tilde{H}_{i}^Tb\leq0.$ Then, we have
$G_{i}^TAx=\tilde{H}_{i}^TGx\leq\tilde{H}_{i}^Tb\leq0.$
According to Lemma \ref{lemmap1}, we have that $\mathcal{P}$
is an invariant set for the continuous system.

Now we consider the ``only if" part. According to Theorem \ref{songthm}, we have that there
exists a $\hat\tau>0$, such that $\mathcal{P}$ is also an invariant
set for the discrete system $ x_{k+1}=(I+A\Delta t)x_k,$ for
every $ 0\leq \Delta t\leq\hat\tau. $  Then, according to Theorem
\ref{polythm}, there exists a matrix $H(\Delta t)\geq0$, such that $H(\Delta t) G=G(I+A\Delta t), \text{ and } H(\Delta t) b\leq b$, i.e., 
\begin{equation}\label{t1}
\frac{H(\Delta t)-I}{\Delta t} G=GA, \text{ and } \frac{H(\Delta t)-I}{\Delta t}b\leq 0.
\end{equation}
Clearly $\tilde{H}=\frac{H(\Delta t)-I}{\Delta t}$ for $\Delta t>0$ satisfies this theorem.  
\end{proof}

We consider the invariance condition of the polyhedron in the form of (\ref{poly2}) for the continuous system (\ref{dyna1}). For an arbitrary convex set in $\mathbb{R}^n$, we have the following conclusion\footnote{We thank the referee for proposing this simple and more transparent proof.}.

\begin{lemma}\label{lemma:general}
Let $\mathcal{S}$ be a convex set in $\mathbb{R}^n$. For any $\ell\in \mathbb{N}$ and   $x, y^1,y^2,...,y^\ell\in\mathcal{S}$ satisfying  $x=\sum_{i=1}^\ell\beta_iy^i,$ where $\sum_{i=1}^\ell \beta_i=1$ and $\beta_i>0$ for every $i\in \mathcal{I}(\ell),$   we have $\mathcal{T}_{\mathcal{S}}(y^i)\subseteq\mathcal{T}_{\mathcal{S}}(x)$ for every $i\in \mathcal{I}(\ell).$
\end{lemma}

\begin{proof} We denote cone$(x,\mathcal{S})=\{\alpha(y-x)\,|\,y\in \mathcal{S}, \alpha\geq0\}$, then we have that $\mathcal{T_S}(x)$ is the same as the topological closure of cone$(x,S)$. Let  $\Phi(x)$ denote the face of $\mathcal{S}$ generated by $x$, i.e., the set $\{y\in \mathcal{S}\,|\,\mu x+(1-\mu)y\in \mathcal{S} \text{ for some }\mu >1\}.$ We first show that for any $x,u\in\mathcal{S}$, if $u\in \Phi(x)$, then $\mathcal{T_S}(u)\subseteq\mathcal{T_S}(x)$. In fact, by definition of $\Phi(x)$ there exists $\mu>1$, such that $v:=\mu x+(1-\mu)u\in \mathcal{S}.$ Then we have $x=(1-\alpha)u+\alpha v$ for some $\alpha,0<\alpha<1.$ Note that for any $y\in \mathcal{S}$, we have $(1-\alpha)y+\alpha v\in \mathcal{S}$ and $[(1-\alpha)y+\alpha v]-x=(1-\alpha)(y-u).$ It follows that cone$(u,\mathcal{S})\subseteq$cone$(x,\mathcal{S})$. By taking the closure of both sides, we have $\mathcal{T_S}(u)\subseteq\mathcal{T_S}(x)$. Since $\sum_{i=1}^\ell \beta_i=1$ and $\beta_i>0$ for every $i\in \mathcal{I}(\ell)$,  $y^i\in \Phi(x)$, for every $i\in \mathcal{I}(\ell)$ we have $y^i\in \Phi(x)$,   the lemma follows immediately. 
\end{proof}

For the polyhedron $\mathcal{P}$ given as in (\ref{poly2}), a vertex of $\mathcal{P}$ is given as   $x^i,$ for some $ i\in \mathcal{I}(\ell_1)$, and an extreme ray of $\mathcal{P}$ is represented as $x^i+\alpha\hat{x}^j, \alpha>0,$ for some $i\in \mathcal{I}(\ell_1)$ and $j\in \mathcal{I}(\ell_2).$
Applying Lemma \ref{lemma:general} to $\mathcal{P}$, we have the following Corollary
\ref{lemma36} about the relationship between tangent cones at a vector and the vertices and extreme rays of $\mathcal{P}$. Note that
$\mathcal{T_P}(x)=\mathbb{R}^n$ for every $x\in \text{int}(\mathcal{S})$, thus Corollary \ref{lemma36} is only nontrivial for $x\in \partial\mathcal{P}.$

\begin{corollary}\label{lemma36}
Let a polyhedron $\mathcal{P}$ be given as in (\ref{poly2}), and $x\in \mathcal{P}$ be a point in $\mathcal{P}$ given as in formula (\ref{poly2}).
Let $\mathcal{I}_1=\{i\in\mathcal{I}(\ell_1)\,|\,\theta_i>0\}$ and $\mathcal{I}_2=\{j\in\mathcal{I}(\ell_2)\,|\,\hat\theta_j>0\}$. Then $\mathcal{T_P}(x^i)\subseteq\mathcal{T_P}(x)$ and $\mathcal{T_P}(x^i+\alpha\hat {x}^j)=\mathcal{T_P}(x^i+\hat {x}^j)\subseteq\mathcal{T_P}(x)$
for $i\in \mathcal{I}_1, j\in \mathcal{I}_2$, and $\alpha>0$, where $x^i+\alpha\hat {x}^j$ is an extreme  ray of $\mathcal{P}$.
\end{corollary}

Let us consider a polytope $\mathcal{\tilde{P}}$ generated by $\{x^1,x^2,...,x^{\ell_1}\}$ as its vertices. Then, according to \cite{blan6}, we have that $\mathcal{T_{\mathcal{\tilde{P}}}}(x^i)$ can be generated as a conic combination of $x^p-x^i$ for all $p\in \mathcal{I}(\ell_1)$, i.e.,  
$\mathcal{T_{\mathcal{\tilde{P}}}}(x^i)=\{y|y=\sum_{p=1, p\neq i}^{\ell_1}\alpha_p(x^p-x^i),\alpha_p\geq0\}.$
Let $\alpha_i=\sum_{p=1,p\neq i}^{\ell_1}\alpha_p.$ Then we have  
$$\mathcal{T_{\mathcal{\tilde{P}}}}(x^i)=\Big\{y|y=\sum_{p=1}^{\ell_1}\alpha_px^p, \alpha_p\geq0, p\neq i, \sum_{p=1}^{\ell_1}\alpha_p=0\Big\}.$$ 
By a similar argument, we have that the exact representations  of the tangent cones at vertices or extreme rays of $\mathcal{P}$ given as in (\ref{poly2}) are presented in Lemma \ref{lem:BasicThm-TanCone} below.

\begin{lemma}\label{lem:BasicThm-TanCone}
Let a polyhedron $\mathcal{P}$ be given as in (\ref{poly2}), and $\mathcal{I}_1'=\{i\in \mathcal{I}(\ell_1)\,|\,$ for any $j\in \mathcal{I}(\ell_2), x^i+\hat{x}^j \text{ is not an extreme ray.}\}$, $\mathcal{I}_1''=\mathcal{I}(\ell_1)\backslash\mathcal{I}_1',$ then

1). For every $i\in\mathcal{I}_1',$ we have
$\mathcal{T}_{\mathcal{P}}(x^i)=\{y\in\mathbb{R}^n\,|\, y=\sum_{p=1}^{\ell_1}\alpha_px^p, \alpha_p\geq0, p\neq i, \sum_{p=1}^{\ell_1}\alpha_p=0\}.$

2). For every $i\in\mathcal{I}_1'',$ we have
$\mathcal{T}_{\mathcal{P}}(x^i)=\{y\in\mathbb{R}^n\,|\, y=\sum_{p=1}^{\ell_1}\alpha_px^p+\sum_{q=1}^{\ell_2}\hat{\alpha}_q \hat{x}^q, \alpha_p, \hat{\alpha}_q\geq0, p\neq i, \sum_{p=1}^{\ell_1}\alpha_p=0\}.$

3). For every $i\in\mathcal{I}_1''$ and $j\in\mathcal{I}(\ell_2)$ such that $x^i+\hat{x}^j$ is an extreme ray, we have
$\mathcal{T}_{\mathcal{P}}(x^{i}+\hat{x}^j)=\{y\in\mathbb{R}^n\,|\, y=\sum_{q=1}^{\ell_2}\hat{\alpha}_q \hat{x}^q, \hat{\alpha}_q\geq0, j\neq q\}.$
\end{lemma}


\begin{lemma}\label{lemma34}
Let $\mathcal{C}$ be a closed convex cone. If $x+\alpha y\in \mathcal{C}$ for all $\alpha>0,$ then $x,y\in \mathcal{C}.$
\end{lemma}

The following lemma presents an invariance condition for a polyhedron in the form of (\ref{poly2}) for the continuous system (\ref{dyna1}).
\begin{lemma}\label{lemmap15}
Let a polyhedron $\mathcal{P}$ be given as in (\ref{poly2}). Then $\mathcal{P}$ is an invariant set for the continuous system (\ref{dyna1}) if and
only if $Ax^i\in\mathcal{T}_{\mathcal{P}}(x^i)$ and $A\hat{x}^j\in\mathcal{T}_{\mathcal{P}}(x^i+\hat{x}^j)$ for $i\in \mathcal{I}(\ell_1)$ and $j\in \mathcal{I}(\ell_2)$, where $x^i+\alpha \hat{x}^j$ for $\alpha\geq0$ is an extreme ray of $\mathcal{P}.$
\end{lemma}
\begin{proof}
We first consider the ``only if" part. According to Nagumo's Theorem \ref{nagumo},  for any $i\in \mathcal{I}(\ell_1)$ and $j\in \mathcal{I}(\ell_2)$ when $x^i+\alpha\hat{x}^j$ for $ \alpha\geq0$ is an extreme ray, we have $Ax^i\in\mathcal{T}_{\mathcal{P}}(x^i)$ and $A(x^i+\alpha\hat{x}^j)\in\mathcal{T}_{\mathcal{P}}(x^i+\hat{x}^j)$. By Lemma \ref{lemma34},  this implies that $A\hat{x}^j\in\mathcal{T}_{\mathcal{P}}(x^i+\hat{x}^j)$. 

For the ``if" part, we choose $x\in \mathcal{P}.$ We represent  $x$ as $x=\sum_{i\in\mathcal{I}_1}\theta_i
x^i+\sum_{j\in\mathcal{I}_2}\hat{\theta}_j\hat{x}^j,$ where
$\mathcal{I}_1=\{i\in\mathcal{I}(\ell_1)\,|\,\theta_i>0\}$ and $\mathcal{I}_2=\{j\in\mathcal{I}(\ell_2)\,|\,\hat\theta_j>0\}.$ Then according to Corollary \ref{lemma36}, we have
$Ax=\sum_{i\in\mathcal{I}_1}\theta_i
Ax^i+\sum_{j\in\mathcal{I}_2}\hat{\theta}_jA\hat{x}^j$. Note that $Ax^i\in \mathcal{T}_{\mathcal{P}}(x^i)\subseteq\mathcal{T}_{\mathcal{P}}(x)$ and $A\hat{x}^j\in\mathcal{T}_{\mathcal{P}}(x^i+\hat{x}^j)\subseteq\mathcal{T}_{\mathcal{P}}(x)$. Since $\mathcal{T}_{\mathcal{P}}$ is a convex cone, it is closed under vector addition. So we have $Ax\in \mathcal{T}_{\mathcal{P}}(x).
$ Finally, the ``if" part follows by Nagumo's Theorem \ref{nagumo}.
\end{proof}

 By Lemma \ref{lem:BasicThm-TanCone}
and Lemma \ref{lemmap15}, the following corollary is immediate.
\begin{corollary}\label{lemmap2}
Let a polyhedron $\mathcal{P}$ be given as in  (\ref{poly2}). Then $\mathcal{P}$ is an invariant set for the continuous system (\ref{dyna1}) if and
only if for  $x^i$, $i\in \mathcal{I}(\ell_1)$, there exist $\alpha_p^i,\hat\alpha_q^i\geq0$ for $p\neq i$,  $\alpha_i^i\leq0$, such that
\begin{equation}\label{eq52}
Ax^i =\sum_{p=1}^{\ell_1} \alpha_p^i x^p+\sum_{q=1}^{\ell_2}\hat{\alpha}_q^i \hat{x}^q, \text{ and } \sum_{p=1}^{\ell_1}\alpha_p^i=0,
\end{equation}
for $\hat{x}^j$, $j\in \mathcal{I}(\ell_2)$,  there exist $ \hat\alpha_q^j\geq0$ for $q\neq j$, and $\hat\alpha_j^j\in \mathbb{R},$ such that
$A\hat{x}^j=\sum_{q=1}^{\ell_2}\hat{\alpha}_q^j \hat{x}^q.$
\end{corollary}

\begin{theorem}\label{polythmcon2}
A polyhedron $\mathcal{P}$ given as in (\ref{poly2}) is an
invariant set for the continuous  system (\ref{dyna1}) if and only
if there exists a matrix $\tilde{L}\in \mathbb{R}^{(\ell_1+\ell_2)\times (\ell_1+\ell_2)}$, such that $\tilde{L}{\geq}_o0$, 
$X\tilde{L}=AX$, and $ \bar{1}\tilde{L}= \bar{0},$ where
$X=[x^1,...,x^{\ell_1},\hat{x}^1,...,\hat{x}^{\ell_2}]$,
$\bar{1}=[1_{\ell_1},0_{\ell_2}].$
\end{theorem}

\begin{proof} This proof is similar to the one given in Theorem \ref{polythmcon}. We denote the $i$-th
column of $\tilde{L}$ by $(l_{1,i},...,l_{\ell_1+\ell_2,i})^T$.

For the ``if" part, we consider $x^i$ with $i\in \mathcal{I}(\ell_1) $.  Since $\tilde{L}{\geq}_o0$, $X\tilde{L}=AX$, and $\bar{1}\tilde{L}= \bar{0}$, we have
$
    Ax^i=\sum_{p=1}^{\ell_1}l_{p,i}x^i+\sum_{q=1}^{\ell_2}l_{\ell_1+q,i}\hat{x}^q, \text{ with } \sum_{p=1}^{\ell_1}l_{p,i}=0, \text{ and } l_{p,i}\geq0, \text{ for } p\neq i.
$
The argument for $\hat{x}^j$ with $j\in\mathcal{I}(\ell_2)$ is similar. Then, according to Corollary  \ref{lemmap2}, we have that $\mathcal{P}$ is an invariant set for the continuous system.

For the ``only if" part, the proof is similar to the one in Theorem \ref{polythmcon}. According to Theorem \ref{songthm} and Theorem \ref{polythm2}, we know that there exists a nonnegative 
matrix $L(\Delta t)$ and a scalar $\hat\tau>0$,  such that
$XL(\Delta t)=(I+\Delta t A)X,~~\bar{1}L(\Delta t)= \bar{1},
\text{ for } 0\leq \Delta t \leq \hat\tau,$ i.e.,
\begin{equation*}\label{t3}
X\frac{L(\Delta t)-I}{\Delta t}=AX,~~\bar{1}\frac{L(\Delta t)-I}{\Delta t}= \bar{0}.
\end{equation*}
Let $\tilde{L}=\frac{L(\Delta t)-I}{\Delta t}$, the theorem is immediate. 
\end{proof}

Since the invariance conditions for a polyhedral cone given in the two different forms can be obtained
by  similar discussions as above,  we only present these invariance
conditions without providing the proofs.

\begin{corollary}\label{coro:polycone}
1). A polyhedral cone $\mathcal{C_P}$ given as in (\ref{polycone1})
is an invariant set for the continuous  system (\ref{dyna1}) if and
only if there exists a  matrix $\tilde{H}\in \mathbb{R}^{m\times m}$, such
that $\tilde{H}\geq_o0 $ and $\tilde{H}G=GA$.

2).A polyhedral cone $\mathcal{C_P}$ given as in (\ref{polycone2})
is an invariant set for the continuous  system (\ref{dyna1}) if and
only if there exists a matrix $\tilde{L}\in \mathbb{R}^{\ell\times \ell}$,  such
that $\tilde{L}\geq_o0 $ and $X\tilde{L}=AX$, where $X=[\hat{x}^1,...,\hat{x}^{\ell}]$.
\end{corollary}

According to Theorem \ref{polythmcon2} and Corollary \ref{coro:polycone}, verifying if a polyhedron given as in (\ref{poly2}) or polyhedral cone given as in (\ref{polycone2}) is an invariant set for the continuous system (\ref{dyna1}) can be done by solving a series of  linear optimization problems. 

\subsection{Ellipsoids}
In this section, we consider the invariance condition for ellipsoids  which are represented by a quadratic inequality.

\subsubsection{Invariance Conditions for Discrete Systems}
The \textit{S}-lemma and Proposition
\ref{prop1} are our main tools to obtain the invariance condition of an ellipsoid for a discrete system.
First, we present a technical lemma.

\begin{lemma}\label{lemmaellip}
Let $Q$ be an $n\times n$ real symmetric matrix and let $\alpha$ be a given real number. Then  $x^TQx\geq \alpha$ for all $x\in
\mathbb{R}^n$  if and only if $Q\succeq0$, and $\alpha\leq0.$
\end{lemma}

\begin{theorem}\label{ellipthm}
An ellipsoid $\mathcal{E}$ given as in (\ref{elli}) is an invariant
set for the discrete  system (\ref{dyna2}) if and only if
\begin{equation}\label{ellip3}
\exists\, \mu \in [0,1], \text{ such that } B_k^TQB_k-\mu Q\preceq 0.
\end{equation}
\end{theorem}

\begin{proof} According to
Proposition \ref{prop1}, to prove this theorem is
equivalent to prove $\mathcal{E}\subseteq \mathcal{E}'$, where
$
\mathcal{E}=\{x\,|\,x^TQx\leq 1\}$ and $
\mathcal{E}'=\{x\,|\,x^TB_k^TQB_kx\leq 1\}.
$
Clearly, $\mathcal{E}\subseteq \mathcal{E}'$ holds if and only if the
following inequality system has no solution:
\begin{equation}\label{ellip2}
-x^TB_k^TQB_kx+1<0, ~~x^TQx-1\leq 0.
\end{equation}
Note that the left sides of the two inequalities in (\ref{ellip2})
are both quadratic functions, thus, according to the
\textit{S}-lemma, we have that (\ref{ellip2}) has no solution is equivalent
to that there exists $\mu\geq0$, such that $-x^TB_k^TQB_kx+1+\mu(x^TQx-1)\geq 0, $ or equivalently,
\begin{equation}\label{for11}
 x^T(\mu
Q-B_k^TQB_k)x\geq \mu-1,~~ \text{ for all } x\in \mathbb{R}^n.
\end{equation}
The theorem follows by applying Lemma \ref{lemmaellip} to
(\ref{for11}).
\end{proof}

We can also consider an ellipsoid as an invariant set for a system in the following perspective. Invariance of a bounded set for a system is possible only if
the system is non-expansive, which  means that for discrete system
(\ref{dyna2}), all eigenvalues of $B_k$ are in a closed unit
disc of the complex plane. Then it becomes clear that
(\ref{ellip3}) has a solution only if  (\ref{dyna2}) is
non-expansive, i.e., the trajectory of (\ref{dyna2}) is non-expansive. One can conclude from this that there is an invariant
ellipsoid for (\ref{dyna2}) if and only if (\ref{ellip3}) has
a solution for a positive definite $Q$. 

Moreover, we can also observe that the smallest $\mu$
solving (\ref{ellip3}) is the largest eigenvalue of $WA^TQAW$, where
$W$ is the symmetric positive definite square root of $Q^{-1}$, i.e., $W^2=Q^{-1}.$

We now present two examples such that condition (\ref{ellip3}) does not hold for $\mu\not\in[0,1]$.  First, let $Q$ be positive definite and $\mu<0$,  then $B_k^TQB_k-\mu Q$ is always a positive definite matrix. Thus condition (\ref{ellip3}) does not hold. Second, let $Q$ be positive definite and  $\mu>1$, consider the
discrete system $x_{k+1}=-x_k.$ One can prove that $\{x\,|\,x^TQx\leq1\}$ is an invariant set for this discrete system. However, in this case, we have $B_k^TQB_k-\mu Q=(1-\mu)Q$, which is always a
negative definite matrix. Thus condition (\ref{ellip3}) does not hold either. 

Apart from  the  simplicity, another advantage of the approach
given in the proof of Theorem \ref{ellipthm} is that it obtains a
sufficient and necessary condition. Also, this approach
highlights the close relationship between the theory of invariant
sets and the Theorem of Alternatives, which is a fundamental result
in optimization community.

\begin{corollary}
Condition (\ref{ellip3}) holds if and only if
\begin{equation}\label{cor1}
\exists\, \nu\in [0,1], \text{ such that }\tilde{Q}=\left(
  \begin{array}{cc}
    Q^{-1} & B_k \\
    B_k^T & \nu Q \\
  \end{array}
\right)\succeq0.
\end{equation}
\end{corollary}

\begin{proof} First, $Q\succ0$ yields $Q^{-1}\succ0$.  By Schur's lemma \cite{boyd2}, $\tilde{Q}\succeq0$ if and only if its Schur complement $\nu
Q-B_k^T(Q^{-1})^{-1}B_k=\nu Q-B_k^TQB_k\succeq0$, i.e.,
(\ref{ellip3}) holds.
\end{proof}

\begin{corollary}\label{cor2}
Condition (\ref{ellip3}) holds if and only if
\begin{equation}\label{ellip4}
B_k^TQB_k-Q\preceq0.
\end{equation}
\end{corollary}
\begin{proof} The ``if" part is immediate by letting $\mu=1$ in
(\ref{ellip3}). For the ``only if'' part, we let $\nu=1-\mu$,
which, by reformulating (\ref{ellip3}), yields
$B_k^TQB_k-Q\preceq -\nu Q\preceq 0,\text{ for } \nu \in [0,1],$
where the second $``\preceq$'' holds due to the fact that $\nu\geq0$
and $Q\succ0$.
\end{proof}

The left side of (\ref{ellip4}) is called
the Lyapunov operator \cite{boyd} in discrete form or Stein transformation \cite{stein11} in dynamical system. Corollary \ref{cor2}
is consistent with the invariance condition of an ellipsoid for discrete system given in
\cite{Blanchini, boyd}. The invariance
condition presented in \cite{Blanchini} is the same as (\ref{ellip4})
without the equality. This is since contractivity
rather than invariance of a set for a system is analyzed in
\cite{Blanchini}.  Lyapunov method is used to
derive condition (\ref{ellip4}) in \cite{boyd}. Apparently, condition
(\ref{ellip4}) is easier to apply than condition (\ref{ellip3}), since
the former one involves only about
the ellipsoid and the system.

The attentive reader may observe that the positive definiteness assumption for matrix Q is never used
in the proof of Theorem \ref{ellipthm}. That assumption was only needed to ensure that the set $\mathcal{S}$ is convex.
Recall that the quadratic functions in the \textit{S}-lemma are not necessarily convex, thus we can extend Theorem
3.16 to general sets which are represented by a quadratic inequality.
\begin{theorem}\label{genequa}
A set $\mathcal{S}=\{x\in\mathbb{R}^n\,|\,x^TQx\leq 1\}$, where $Q\in \mathbb{R}^{n\times n},$ is an invariant
set for the discrete  system (\ref{dyna2}) if and only if
\begin{equation}\label{ellip32d}
\exists\, \mu \in [0,1], \text{ such that } B_k^TQB_k-\mu Q\preceq 0.
\end{equation}
\end{theorem}

The proof of Theorem \ref{genequa} is the same as that of Theorem 3.16, 
so we do not duplicate that proof here. A trivial example that satisfy  
the condition in is given by choosing $Q$ to be any indefinite matrix, 
$B_k = I$,  and we choose $\mu = 1$. It is easy to see that for this choice 
condition (24) holds. Further exploring the implications of possibly using 
nonconvex and unbounded invariant sets is far from the
main focus of our paper, so this topic remains the subject of further research.

\subsubsection{Invariance Conditions for Continuous Systems}
We first present an interesting result about the solution of continuous system.

\begin{proposition}
The solution of the continuous  system (\ref{dyna1}) is on the
boundary of the ellipsoid $\mathcal{E}$ given as in (\ref{elli})
(or the Lorenz cone $\mathcal{C_L}$ given as in (\ref{ellicone})) whenever $x_0\in \partial\mathcal{E}$(or $x_0\in\partial \mathcal{C_L}$)
if and only if
\begin{equation}\label{cond111}
\sum_{i=0}^{k-1}\frac{1}{(k-1)!}\binom{k-1}{i}(A^{i})^TQA^{k-i-1}=0, \text{ for }
k=2,3,....
\end{equation}
\end{proposition}
\begin{proof}
We consider only ellipsoids, and the proof for Lorenz cones is similar. The solution of (\ref{dyna1}) is given as $x(t)=e^{At}x_0,$ thus
$x(t)\in
\partial\mathcal{E} $ if and only if $x_0^T(e^{At})^TQe^{At}x_0=1,$
which can be expanded, by substituting  $e^{At}=\sum_{i=0}^\infty
\frac{A^i}{i!}t^i$, as
\begin{equation*}\label{cond12}
\sum_{k=1}^\infty t^{k-1}x_0^T\tilde{Q}_{k-1}x_0=1, \text{ where }
\tilde{Q}_{k-1}=\sum_{i=0}^{k-1}\frac{1}{(i)!(k-i-1)!}(A^{i})^TQA^{k-i-1},
\end{equation*}
for any $x_0^TQx_0=1$ and $t\geq0.$ Thus, $\tilde{Q}_{k-1}=0,$
for $k\geq2.$ Also, note that
$\frac{1}{(k-1)!}\binom{k-1}{i}=\frac{1}{(i)!(k-i-1)!}$, condition
(\ref{cond111}) is immediate.
\end{proof}

In particular, when $k=2$, condition
(\ref{cond111}) yields $A^TQ+QA=0$. The left hand side of this
equation is called Lyapunov operator in continuous form. 
The following invariance conditions is first given by Stern and Wolkowicz
\cite{stern}, where they  consider only Lorenz cones and their proof is using  the
concept of cross-positivity. Here we present a simple proof.

\begin{lemma}\label{lemma4}
\emph{\cite{stern}} An ellipsoid $\mathcal{E}$ given in the form of  (\ref{elli})
(or a Lorenz cone $\mathcal{C_L}$ given in the form of  (\ref{ellicone}))
is an invariant set for the continuous system (\ref{dyna1}) if and
only if
\begin{equation}\label{no1}
(Ax)^TQx\leq 0, \text{ for all } x\in \partial \mathcal{E} ~( \text{
 or }x\in
\partial \mathcal{C_L}).
\end{equation}

\end{lemma}
\begin{proof}
We  consider only ellipsoids, and the proof is  analogous for Lorenz cones. Note that $\partial
\mathcal{E}=\{x\,|\,x^TQx=1\}$, thus the outer normal vector of $\mathcal{E}$ at $x\in \partial \mathcal{E}$ is $Qx$. Then we have $\mathcal{T}_{\mathcal{E}}(x)=\{y\,|\, y^TQx\leq0\}$, thus this
theorem follows by Theorem \ref{nagumo}.
\end{proof}

We
now present a sufficient and necessary  condition that an ellipsoid is invariant for the continuous system.

\begin{theorem}\label{ellipthmcon1}
An ellipsoid $\mathcal{E}$ given as in (\ref{elli}) is an
invariant set for the continuous  system (\ref{dyna1}) if and only
if
\begin{equation}\label{ellip3con1}
A^TQ+QA\preceq 0.
\end{equation}
\end{theorem}
\begin{proof}
According to Lemma \ref{lemma4}, we have that condition (\ref{no1}) holds, i.e., $\mathcal{E}$ is an invariant set for the continuous system if and only if 
\begin{equation}\label{ellipeq1}
x^T(A^TQ+QA)x\leq 0, \text{ for all } x\in \partial \mathcal{E}. 
\end{equation}
Clearly (\ref{ellip3con1}) implies (\ref{ellipeq1}). Now assume (\ref{ellipeq1}) holds. Then 
for all nonzero  $y\in \mathbb{R}^n,$ there exists an $x\in\partial
\mathcal{E}$ and $\gamma>0$, such that $y=\gamma x.$ Then $
y^T(A^TQ+QA)y=\frac{1}{\gamma^2}x^T(A^TQ+QA)x\leq 0, $ which yields
condition (\ref{ellip3con1}). 
\end{proof}

The presented method in the proof of Theorem \ref{ellipthmcon1} is simpler than the traditional
Lyapunov method to derive the invariance condition. However, the approach in the proof cannot be
used for Lorenz cones, since the origin is not in the
interior of Lorenz cones.

\subsection{Lorenz Cones}
A Lorenz cone $\mathcal{C_L}$ given as in (\ref{ellicone})
also has a quadratic form, but the way to obtain the
invariance condition of a Lorenz cone for discrete system is much
more complicated than that for an ellipsoid. The difficulty is mainly due to
the existence of the second constraint in (\ref{ellicone}).

\subsubsection{Invariance Conditions for Discrete Systems}
The representation of the nonconvex set $\mathcal{C_L}\cup (-\mathcal{C_L})=\{x\,|\,x^TQx\leq0\}$ involves only the quadratic form, which is almost the same as an ellipsoid. We can first derive
the invariance condition of this set for discrete
system. Recall that the \textit{S}-lemma does not require that the quadratic
functions have to be convex, thus the
\textit{S}-lemma  is still valid for the nonconvex set.

\begin{theorem}\label{them31}
The nonconvex set $\mathcal{C_L}\cup (-\mathcal{C_L})$ is an
invariant set for the discrete  system (\ref{dyna2}) if and only if
\begin{equation}\label{ellip32}
\exists\, \mu \geq 0, \text{ such that } B_k^TQB_k-\mu Q\preceq 0.
\end{equation}
\end{theorem}

\begin{proof} The proof is closely following the ideas in the proof of Theorem \ref{ellipthm}.
The only difference is that the right side in (\ref{for11}) is 0 rather than
$1-\mu$, which is why the condition $\mu\leq1$ is absent in this
case.
\end{proof}

The invariance condition for $\mathcal{C_L}\cup (-\mathcal{C_L})$
shown in (\ref{ellip32}) is similar to the one proposed by Loewy and
Schneider in \cite{loewy}.  They proved by contradiction using the properties of copositive matrices that when the rank of $A$ is greater than 1, $B_k\mathcal{C_L}\subseteq\mathcal{C_L}$ or $-B_k\mathcal{C_L}\subseteq\mathcal{C_L}$ if and only if (\ref{ellip32}) holds.  They also concluded (see \cite[Lemma 3.1]{loewy}) that when the rank of  $B_k$ is 1, $B_k\mathcal{C_L}\subseteq \mathcal{C_L}$ if and only if there exist two vectors $x,y\in \mathcal{C_L}$,
such that $B_k=xy^T$. 

The following example shows that for the given $B_k$ and $Q$, only $\mu=1$ satisfies condition (\ref{ellip32}).
Let $B_k=Q=\text{{diag}}\{1,...,1,-1\}$. Then the Lorenz cone is an invariant set for the
system, since such a Lorenz cone
is a self-dual cone\footnote{A self-dual cone is a cone that
coincides with its dual cone, where the dual cone for a cone
$\mathcal{C}$ is defined as $\{y\;|\;x^Ty\geq0, \forall
x\in\mathcal{C}\}$. }. The left hand side in
(\ref{ellip32}) is, however, now simplified to $(1-\mu)Q$ which is 
 negative semidefinite only for $\mu=1,$ because inertia$\{Q\}=\{n-1,0,1\}$.

In the case of ellipsoids, we used Schur's  lemma, see, e.g.,  \cite{roos}, to simplify invariance condition (\ref{ellip3}) to (\ref{cor1}), which was further simplified to the parameter free invariance condition (\ref{ellip4}). Although conditions (\ref{ellip3}) and (\ref{ellip32}) are similar, it seems to be impossible to develop a parameter free condition analogous to (\ref{ellip4}) for Lorenz cone. This is since  matrix $Q$ for a Lorenz cone is neither positive nor
negative semidefinite. 

To find the scalar $\mu$ in (\ref{ellip32}) is essentially a semidefinite optimization (SDO) problem. Various
celebrated SDO solvers, {e.g.}, SeDuMi \cite{sedumi},
CVX \cite{cvx}, and SDPT3 \cite{sdpt3}
have been shown robust performance in solving a SDO problems numerically.

\begin{corollary}\label{cor4}
If $\lambda_1(B_k^TQB_k)\leq0$, then the Lorenz cone $\mathcal{C_L}$ given as in (\ref{ellicone}) is an invariant set for the discrete
system (\ref{dyna2}).
\end{corollary}

Corollary \ref{cor4} gives a simple sufficient condition such that a Lorenz cone is an invariant set. However, one must note that this result is valid only if matrix $B_k$ is singular. Let  $\text{dim}(M)$ represent the dimension of a matrix  $M$. In fact, if $B_k$ is nonsingular, then by
Sylvester's law of inertia \cite{horn}, we have that $\lambda_1(B_k^TQB_k)>0$. When $\lambda_1(B_k^TQB_k)\leq0$, we have that the rank of $B_k$ is at most 1. This is because, if the rank is larger than 1, then range$(B_k)\cap\text{span}\{u_1,u_2,...,u_{n-1}\}$ must be a nonzero subspace. 
This is  since the sum of $\text{dim}(\text{range}(B_k))$ and $\text{dim}(\text{span}\{u_1,u_2,...,u_{n-1}\})$ is greater than or equal to $ 2+(n-1)=n+1>n,$ and $\text{dim}(\text{range}(B_k))\cap\text{span}\{u_1,u_2,...,u_{n-1}\})$ is greater than or equal to 1.
Then let $0\neq x\in$ range$(B_k)\cap\text{span}\{u_1,u_2,...,u_{n-1}\}$, we have $x^T(B_k^TQB_k)x>0$, which contradicts to $\lambda_1(B_k^TQB_k)\leq 0.$ Also, note that the rank of $B_k^TQB_k$ is less than or equal to the minimum of the rank of $B_k$ and the rank of $Q$, so if  $B_k^TQB_k$ is not zero matrix and $\lambda_1(B_k^TQB_k)\leq0$, then the rank of  $B_k$ is equal to 1. 

The interval of the scalar $\mu$ in (\ref{ellip32}) can be tightened by incorporating the eigenvalues and
eigenvectors of $Q$. Such a tighter condition is presented in
Corollary \ref{cor51}.

\begin{corollary}\label{cor51}
If $\mu$ satisfies $B_k^TQB_k-\mu Q\preceq0$, then
\begin{equation}\label{con277}
\max\Big\{0,\max_{1\leq i\leq
n-1}\Big\{\frac{u_i^TB_k^TQB_ku_i}{\lambda_i}\Big\}\Big\}\leq
\mu\leq \frac{u_n^TB_k^TQB_ku_n}{\lambda_n}.
\end{equation}
\end{corollary}
\begin{proof}
Multiplying condition (\ref{ellip32}) by $u_i^T$ from the left and
$u_i$ from the right, we have $u_i^TB_k^TQB_ku_i-\mu u_i^TQu_i\leq 0$.
Since $u_i^TQu_i=\lambda_i u_i^Tu_i=\lambda_i>0$, for $i\in
\mathcal{I}(n-1)$, and $u_n^TQu_n=\lambda_n<0$, condition
(\ref{con277}) follows immediately.
\end{proof}

Corollary \ref{cor51} presents tighter bounds for the scalar $\mu$
in (\ref{con277}) in terms of an algebraic form.
The existence of a scalar $\mu$ implies  that the upper bound should be no less than
the lower bound in (\ref{con277}). However, this is not
always true. We now present a geometrical interpretation of the
interval of the scalar $\mu$, that can be directly derived from
Corollary \ref{cor51}.

\begin{corollary}\label{cor211}
The relationship between the vector $B_ku_i, $ and the scalars
$u_i^TB_k^TQB_ku_i$, and $\mu$ are as follows:
\begin{itemize}
  \item If $B_ku_n\notin \mathcal{C_L}\cup(- \mathcal{C_L})$, then
 $\mu$ satisfying (\ref{con277}) does not exist.
  \item If $B_ku_i\in \mathcal{C_L}\cup(-
\mathcal{C_L})$ for all $i\in\mathcal{I}(n-1)$, then
\begin{itemize}
  \item if $B_ku_n\in
\partial\mathcal{C_L}\cup(-\partial \mathcal{C_L})$ and (\ref{con277}) holds, then
$\mu=0$.
  \item if $B_ku_n\in \emph{int
}\mathcal{C_L}\cup(- \emph{int }\mathcal{C_L})$ and (\ref{con277}) holds, then $\mu\in
\Big[0, \frac{u_n^TB_k^TQB_ku_n}{\lambda_n}\Big].$
\end{itemize}
   \item Let $\mathcal{I}=\{i\;|\;B_ku_i\notin \mathcal{C_L}\cup(-
\mathcal{C_L})\}$. If the set $\mathcal{I}\subseteq\mathcal{I}(n-1)$
is nonempty,  then
\begin{itemize}
  \item if $B_ku_n\in
\partial\mathcal{C_L}\cup(-\partial \mathcal{C_L})$, then
$\mu$ satisfying (\ref{con277}) does not exist.
  \item if $B_ku_n\in \emph{int
}(\mathcal{C_L})\cup(- \emph{int}(\mathcal{C_L}))$, then
\begin{itemize}
  \item if there exist $i^*\in\mathcal{I}$,
such that $\frac{u_{i^*}^TB_k^TQB_ku_{i^*}}{\lambda_{i^*}}>
\frac{u_n^TB_k^TQB_ku_n}{\lambda_n}$, then  $\mu$ satisfying (\ref{con277}) does not exist.
  \item otherwise, if (\ref{con277}) holds, then \\
$\mu\in\Big[\max_{i\in\mathcal{I}}\Big\{\frac{u_{i}^TB_k^TQB_ku_{i}}{\lambda_{i}}\Big\},
\frac{u_n^TB_k^TQB_ku_n}{\lambda_n}\Big]$.
\end{itemize}
\end{itemize}
\end{itemize}
\end{corollary}

We now consider the invariance condition of a Lorenz cone $\mathcal{C_L}$ given as in (\ref{ellicone}), which is a convex set and can handle expansive systems.

\begin{lemma}\emph{\cite{stern}}\label{lemma123}
A Lorenz cone $\mathcal{C_L}$ given as in (\ref{ellicone}) can be written as $T\mathcal{K}_n$, where $\mathcal{K}_n$ is the standard Lorenz cone and $T$ is the nonsingular matrix, 
\begin{equation}\label{matr}
T=\Big[\frac{u_1}{\sqrt{\lambda_1}},...,\frac{u_{n-1}}{\sqrt{\lambda_{n-1}}},\frac{u_n}{\sqrt{-\lambda_n}}
\Big].
\end{equation}
\end{lemma}

\begin{lemma}\label{lemma31}
A Lorenz cone $\mathcal{C_L}$ given as in (\ref{ellicone}) is
an invariant set for the discrete  system (\ref{dyna2}) if and only
if the standard Lorenz cone $\mathcal{K}_n$ is an invariant set
for the following discrete system
\begin{equation}\label{dyna31}
x_{k+1}=T^{-1}B_kTx_k,
\end{equation}
where $T$ is defined as (\ref{matr}).
\end{lemma}

\begin{proof}
The Lorenz cone $\mathcal{C_L}$  is an invariant set for
(\ref{dyna2}) if and only if $B_k\mathcal{C_L}\subseteq
\mathcal{C_L}$. This holds if and only if
$B_kT\mathcal{K}_n\subseteq T\mathcal{K}_n$, which is equivalent
to $T^{-1}B_kT\mathcal{K}_n\subseteq \mathcal{K}_n$.
\end{proof}

The invariance condition of a Lorenz
cone for discrete systems is presented in Theorem \ref{themcone12}. Although we have developed such invariance condition
independently,  it was brought to our attention recently that the
invariance condition is the same as the one proposed by
Aliluiko and Mazko in \cite{alil}. But our proof is
more straightforward.

\begin{theorem}\label{themcone12}
A Lorenz cone $\mathcal{C_L}$ (or $-\mathcal{C_L}$) given as in
(\ref{ellicone}) is an invariant set for the discrete  system
(\ref{dyna2}) if and only if
\begin{equation}\label{ellip33}
\exists\, \mu\geq0, \text{ such that }B_k^TQB_k-\mu Q\preceq 0,~ u_n^TB_ku_n\geq0, ~ u_n^TB_kQ^{-1}B_k^Tu_n\leq0,
\end{equation}
where $u_n$ is the eigenvector corresponding to the unique negative
eigenvalue $\lambda_n$ of $Q.$
\end{theorem}

\begin{proof}
Since $B_k\mathcal{C_L}\subseteq\mathcal{C_L}$ if and only $B_k(-\mathcal{C_L})\subseteq-\mathcal{C_L},$ we only present the proof for
$\mathcal{C_L}$. For an arbitrary $x\in \mathcal{C_L}$,
by Theorem \ref{them31}, we have that $B_kx\in
\mathcal{C_L}$ or $B_kx\in -\mathcal{C_L}$ if and only if condition
(\ref{ellip32}) is satisfied. To ensure that only
$B_kx\in\mathcal{C_L}$ holds, some additional conditions should be
added.

 According to Lemma \ref{lemma31}, we may consider $\mathcal{K}_n$ and
the discrete system (\ref{dyna31}), where the coefficient matrix,
denoted by $\tilde{A}$,  can be explicitly written as
$$
\tilde{A}=T^{-1}B_kT =\left[
      \begin{array}{ccc}
       u_1^TB_ku_1 & \cdots & \sqrt{-\frac{\lambda_1}{\lambda_n}}u_1^TB_ku_n \\
       \vdots & \ddots & \vdots \\
       \sqrt{-\frac{\lambda_n}{\lambda_1}}u_n^TB_ku_1  & \cdots & u_n^TB_ku_n \\
       \end{array}
       \right].
$$
Then, according to Theorem \ref{them31},  condition (\ref{ellip32})
is equivalent to 
\begin{equation}\label{c1}
\exists\, \mu\geq 0, \text{ such that
}(T^{-1}B_kT)^T\tilde{I}T^{-1}B_kT-\mu \tilde{I}\preceq 0,
\end{equation}
where $\tilde{I}=\text{{diag}}\{1,...,1,-1\}$. Note that
$T^TQT=\tilde{I}$, condition (\ref{c1}) is equivalent to
\begin{equation*}
\exists\, \mu\geq 0, \text{ such that } B_k^TQB_k-\mu Q\preceq 0.
\end{equation*}
Recall that we denote the $i$-th row of a matrix $M$ by
$M_i^T.$ Also, the second constraint in the
formulae of $\mathcal{K}_n$ requires
that for every $x\in \mathcal{K}_n$ the last coordinate
in $x$ is nonnegative. Since  $\tilde{A}\mathcal{K}_n\subseteq\mathcal{K}_n$, we have $ \tilde{A}_n^Tx\geq0, \text{ for all } x\in \mathcal{K}_n.$
Note that $\mathcal{K}_n$ is a self-dual cone, we have $ \tilde{A}_n^Tx\geq0, \text{ for all } x\in \mathcal{K}_n$ if
and only if $\tilde{A}_n\in \mathcal{K}_n$. Now
we have
\begin{equation}\label{row}
\tilde{A}_n^T =\sqrt{-\lambda_n}
\Big(\frac{1}{\sqrt{\lambda_1}}u_n^TB_ku_1,\frac{1}{\sqrt{\lambda_2}}u_n^TB_ku_2,...,\frac{1}{\sqrt{-\lambda_n}}u_n^TB_ku_n\Big)
=\sqrt{-\lambda_n}u_n^TB_kT.
\end{equation}
Substituting the value of $\tilde{A}_n^T$ given by the right side of (\ref{row}) into the first inequality in the formulae
of $\mathcal{K}_n$, we have
\begin{equation}\label{conin}
 -\lambda_n(T^TB_k^Tu_n)^T\tilde{I}(T^TB_k^Tu_n)\leq0.
\end{equation}
Since $\lambda_n<0 $ and
$T\tilde{I}T^T=\sum_{i=1}^n\frac{u_iu_i^T}{\lambda_i}=Q^{-1}$, where
the second equality is due to the spectral decomposition of
$Q^{-1}$, we have that (\ref{conin}) is equivalent to $
u_n^TB_kQ^{-1}B_k^Tu_n\leq0. $ Also, substituting (\ref{row}) into the
second inequality in the formulae of $\mathcal{K}_n$ yields
$u_n^TB_ku_n\geq0.$  The proof is complete.
\end{proof}

\begin{remark}\label{cor6}
The inequality system $u_n^TB_kQ^{-1}B_k^Tu_n\leq0$ and $u_n^TB_ku_n\geq0$ holds if and only if
$u_n^TB_kx\geq0$, for all $x\in \mathcal{C_L}.$
\end{remark}
\begin{proof}
Since $x^TQx\leq0$ can be written as
$x^TU\Lambda^{\frac{1}{2}}\tilde{I}\Lambda^{\frac{1}{2}}U^Tx\leq0$,
we have $x\in \mathcal{C_L}$ if and only if
$\Lambda^{\frac{1}{2}}U^Tx\in \mathcal{K}_n.$ Similarly, since
$Q^{-1}=U\Lambda^{-\frac{1}{2}}\tilde{I}\Lambda^{-\frac{1}{2}}U^T,$
we have $u_n^TB_kQ^{-1}B_k^Tu_n\leq0$ can be written as
$u_n^TB_kU\Lambda^{-\frac{1}{2}}\tilde{I}\Lambda^{-\frac{1}{2}}U^TB_k^Tu_n\leq0$,
which yields $\Lambda^{-\frac{1}{2}}U^TB_k^Tu_n\in \mathcal{K}_n\cup(-\mathcal{K}_n).$
Since the set $\mathcal{K}_n$ is a self-dual cone, we have $
(\Lambda^{-\frac{1}{2}}U^TB_k^Tu_n)^T(\Lambda^{\frac{1}{2}}U^Tx)\geq0 $,
which can be simplified to $ u_n^TB_kx\geq0 $, for all $x\in
\mathcal{C_L}.$
\end{proof}

The normal plane of the eigenvector $u_n$ that contains the origin
separates $\mathbb{R}^n$ into two half spaces.  Corollary \ref{cor6}
presents a geometrical interpretation that $A$ transforms
the Lorenz cone $\mathcal{C_L}$ to the half space that contains
eigenvector $u_n$, {i.e.}, $B_k\mathcal{C_L}\subseteq\{y\,|~u_n^Ty\geq0\}$. Moreover, note that
$u_n^TB_kx=(B_k^Tu_n)^Tx,$ which shows that the vector $B_k^Tu_n$ is in
the dual cone of $\mathcal{C_L}.$

\begin{corollary}\label{cor5}
If condition (\ref{ellip33}) holds, then
\begin{equation}\label{con27}
0\leq \mu\leq \frac{u_n^TB_k^TQB_ku_n}{\lambda_n}.
\end{equation}
\end{corollary}
\begin{proof}
 The proof is analogous to the one given in the proof of Corollary
\ref{cor51}.
\end{proof}

The interval for the scalar $\mu$ in condition (\ref{con27}) is wider
but simpler than the one presented in Corollary \ref{cor51}.
Analogous to Corollary \ref{cor211}, we present an intuitive
geometrical interpretation of $\mu$ for Lorenz cones.

\begin{corollary}\label{prop2}
The relationship between the vector  $B_ku_n, $ and the scalars $
u_n^TB_k^TQB_ku_n$, and $\mu$ are as follows:
\begin{itemize}
  \item If $B_ku_n\notin \mathcal{C_L}\cup(- \mathcal{C_L})$, then
$\mu$ satisfying (\ref{con27}) does not exist.
  \item If $B_ku_n\in \partial\mathcal{C_L}\cup(-\partial
\mathcal{C_L})$ and (\ref{con27}) holds, then $\mu=0$.
  \item If $B_ku_n\in \emph{int }(\mathcal{C_L})\cup(- \emph{int
}(\mathcal{C_L}))$ and (\ref{con27}) holds, then  $\mu\in \Big[0,
\frac{u_n^TB_k^TQB_ku_n}{\lambda_n}\Big].$
\end{itemize}
\end{corollary}

\subsubsection{Invariance Conditions for Continuous Systems} 

Now we consider the invariance condition of Lorenz cones for the continuous system. 
We also need to analyze the eigenvalue of a sum of two symmetric matrices for the invariance conditions for continuous systems. The following lemma is a useful tool in our analysis. It shows a fact that the spectrum of a matrix is stable under a small perturbation by another matrix. Since the statement is obvious, we omit the proof. 

%
\begin{lemma}\label{lemma2}
Let $M$ and $N$ be two symmetric matrices. Then
\begin{itemize}
  \item if there exists a $\hat\tau>0$,
such that $M+\tau N\preceq0,$ for $0<\tau\leq \hat\tau$,
 then $M\preceq0.$
  \item if $M\prec0$, then there exists  a $\hat\tau>0$,
such that $M+\tau N\preceq0,$ for $0<\tau\leq \hat\tau$.
\end{itemize}
\end{lemma}


Similar to the case for discrete system, 
we first consider the invariance condition of the nonconvex set $\mathcal{C_L}\cup (-\mathcal{C_L})$ for the continuous system.

\begin{theorem}\label{ellipthmcon2}
The nonconvex set $\mathcal{C_L}\cup (-\mathcal{C_L})$  is an
invariant set for the continuous system (\ref{dyna1}) if and only if
\begin{equation}\label{ellip32con}
\exists\, \eta\in \mathbb{R}, \text{ such that } A^TQ+QA-\eta Q\preceq 0.
\end{equation}
\end{theorem}

\begin{proof}
For the  ``if" part, i.e., condition (\ref{ellip32con}) holds,
then for every $x\in \partial \mathcal{C_L}\cup
(-\partial\mathcal{C_L})$, we have
$(Ax)^TQx=(Ax)^TQx-\frac{\eta}{2} x^TQx=\frac{1}{2}x^T(A^TQ+QA-\eta
Q)x\leq0.$
Thus, by Lemma \ref{lemma4}, the set
$\mathcal{C_L}\cup (-\mathcal{C_L})$  is an invariant set for continuous system.

Next, we prove the ``only if" part. According to Theorem \ref{songthm},  there exists a $\hat{\tau}>0$,
such that for every $ 0\leq \Delta
t\leq\hat\tau $, $\mathcal{C_L}\cup (-\mathcal{C_L})$ is also an invariant set for $x_{k+1}=(I-A\Delta t)^{-1}x_{k}$. 
By Theorem
\ref{them31} and $ (I-A\Delta t)^{-1}=I+A\Delta
t+A^2\Delta t^2+\cdots $, we have $\exists ~\mu(\Delta t) \geq 0,$  such that 
\begin{equation}\label{ttt21}
\frac{1-\mu(\Delta
t)}{\Delta t}Q+(A^TQ+QA)+\Delta t K(\Delta t)\preceq0,
\end{equation}
where $K(\Delta t)=(A^TQA+(A^{2})^TQ+QA^{2})+\Delta
t((A^{2})^TQA+A^{T}QA^2+(A^{3})^TQ+QA^{3})+\mathcal{O}((\Delta t)^2)$.  Since $Q$ and
$A$ are constant matrices, and applying  the fact that
$\|M\|=\|M^T\|$, $\|M+N\|\leq\|M\|+\|N\|$ and
 $\|MN\|\leq\|M\|\|N\|$, we have
\begin{equation*}
\begin{split}
\|K(\Delta t)\|&\leq\sum_{i=3}^\infty i\|Q\|\|A\|^{i-1}(\Delta
t)^{i-3}=\|Q\|\|A\|^2\sum_{i=0}^\infty(i+3)(\Delta t\|A\|)^i \\
&= \|Q\|\|A\|^2\frac{3-2\Delta t \|A\|}{(1-\Delta t\|A\|)^2}\leq
8\|Q\|\|A\|^2,
\end{split}
\end{equation*}
where $\Delta t\leq \frac{5}{4}\|A\|^{-1}$ such that $(3-2\Delta t
\|A\|)/(1-\Delta t \|A\|)^2\leq 8$. Also, applying the
relationship between spectral radius $\rho(A)$ and its induced norm,
$\rho(A)\leq \|A\|$ (see \cite{Derz}), to $K(\Delta t)$, we have
$$
|\lambda_i(K(\Delta t))|\leq \rho(K(\Delta t)) \leq \|K(\Delta
t)\|\leq 8\|Q\|\|A\|^2, \text{ for } i\in \mathcal{I}(n),
$$
i.e., the eigenvalues of $K(\Delta t)$ are bounded. Let us denote $\eta(\Delta t)=\frac{\mu(\Delta
t)-1}{\Delta t}$. Then (\ref{ttt21}) is rewritten as 
\begin{equation}\label{equ34}
-\eta(\Delta t)Q+A^TQ+QA+K(\Delta t)\Delta t \preceq0.
\end{equation}
By multiplying both sides of (\ref{equ34}) by $u_n$, where $u_n$ is the eigenvector corresponding to the negative eigenvalue $\lambda_n$,  we have 
\begin{equation}
u_n^T(A^TQ+QA)u_n+\Delta t u_n^TK(\Delta t)u_n\leq \eta(\Delta t)\lambda_n.
\end{equation}
Since $K(\Delta t)$ is bounded, we have $\Delta t u_n^TK(\Delta t)u_n\rightarrow 0$ as $\Delta t\rightarrow 0$.  This implies that $\eta(\Delta t)$ is bounded for $0\leq \Delta t\leq \hat\tau$ for some $\hat\tau>0.$ 
Therefore\footnote{Here we use the fact that every bounded sequence has a convergent subsequence, see, e.g.,  \cite{rudin}. }, we can take a subsequence $\{\Delta t_\ell\}$ such that  $\eta(\Delta t_\ell)\rightarrow \eta$ as $\Delta t_\ell\rightarrow 0$, which yields (\ref{ellip32con}).
 The proof is complete.
\end{proof}

The approach in the proof of Theorem \ref{ellipthmcon2} can be also used to prove
Theorem \ref{ellipthmcon1}. The only remaining invariance condition is the one of a Lorenz cone for continuous system.

\begin{theorem}\label{ellipthmcon3}
A Lorenz cone $\mathcal{C_L}$ (or $ -\mathcal{C_L}$)  is an
invariant set for the continuous  system (\ref{dyna1}) if and only
if (\ref{ellip32con}) holds. 
\end{theorem}
\begin{proof}
Consider the continuous system with
 $x_0\in \mathcal{C_L}$,
according to Theorem \ref{ellipthmcon2},  the trajectory $x(t)$ will
stay in $\mathcal{C_L}\cup(-\mathcal{C_L})$ if  condition
(\ref{ellip32con}) is satisfied. If $x(t)$ would move over to
$-\mathcal{C_L}$, then $x(t)$ must go through the origin,
i.e., $x(t^*)=0$ for some $t^*\geq0.$ Note that $x(t)=e^{A(t-t^*)}x(t^*)=0$ for any $t>t^*,$ i.e., the origin is an equilibrium point,
which means $\mathcal{C_L}$ is an invariant set for the continuous
system. Thus the theorem is immediate.
\end{proof}

In fact, a direct proof of Theorem \ref{ellipthmcon3} can be given as follows: one can also prove that the second and third
conditions in (\ref{ellip33}) hold by choosing sufficiently small $\Delta t$. To be specific, for the second condition in
(\ref{ellip33}), we have
\begin{equation}\label{eq131}
u_n^T(I-\Delta t A)^{-1}u_n\geq0,\text{ if and only if
}\|u_n\|^2+\sum_{i=1}^\infty (\Delta t)^iu_n^TA^iu_n\geq 0,
\end{equation}
where the second term, when $\Delta t <\|A\|^{-1},$ can be
bounded as follows: $ \left|\sum_{i=1}^\infty (\Delta
t)^iu_n^TA^iu_n\right|\leq\|u_n\|^2\frac{\Delta t \|A\|}{(1-\Delta t
\|A\|)}. $
Thus, we can choose the time step less than the half of reciprocal of the norm of $A$, i.e., $\Delta
t<0.5\|A\|^{-1}$, such that condition (\ref{eq131}) holds.
Similarly, the third condition in (\ref{ellip33}) can be transformed
to
\begin{equation}\label{eq512}
u_n^T(I-\Delta t A)^{-1}Q^{-1}(I-\Delta t A)^{-T}u_n\leq 0, \text{ if
and only if } \frac{1}{\lambda_n}\|u_n\|^2+ K(\Delta t)\leq0,
\end{equation}
where we use the fact that $u_n$ is the eigenvector corresponding to the eigenvalue 
$ \lambda_n^{-1}$ of $Q^{-1}$, and  $K(\Delta t)=\Delta t
u_n^T(AQ^{-1}+Q^{-1}A^T)u_n+(\Delta
t)^2u_n^T(AQ^{-1}A+A^2Q^{-1}+Q^{-1}A^{2T})u_n+\cdots.$ We note that 
inertia$\{Q\}=\{n-1,0,1\}$ implies
inertia$\{Q^{-1})\}=\{n-1,0,1\},$ then we have that $Q^{-1}$ exists,
which yields the following:  $ |K(\Delta t)|\leq \|u\|^2(2\Delta t
\|A\|\|Q^{-1}\|+3\Delta
t^2\|A\|^2\|Q^{-1}\|+\cdots)=\|u\|^2\|Q^{-1}\|\tfrac{2\Delta t
\|A\|-(\Delta t \|A\|)^2}{(1-\Delta t \|A\|)^2}. $ We can choose
$\Delta t\leq
\min\{0.5\|A\|^{-1},(\|A\|(1-4\lambda_n\|Q^{-1}\|)^{-1}\},$
such that (\ref{eq512}) holds. In fact,
\begin{equation*}
\begin{split}
\frac{1}{\lambda_n}\|u_n\|^2+ K(\Delta t)&\leq
\|u_k\|^2\Big(\frac{1}{\lambda_n}+\|Q^{-1}\|\frac{2\Delta t
\|A\|-(\Delta t \|A\|)^2}{(1-\Delta t
\|A\|)^2}\Big)\\
&\leq\|u\|^2\Big(\frac{1}{\lambda_n}+\|Q^{-1}\|\frac{4\Delta
t \|A\|}{1-\Delta t \|A\|}\Big)\leq0.
\end{split}
\end{equation*}

Condition (\ref{ellip32con}) is the same as the one presented  in
\cite{stern}, whose proof is much more complicated than ours.
Finding the value of $\eta$ in Theorem \ref{ellipthmcon2} and
\ref{ellipthmcon3} is essentially a semidefinite optimization
problem. For example, we can use the following semidefinite optimization problem:
\begin{equation}\label{eqdde}
\max \{ \eta\in \mathbb{R}~|~A^TQ+QA-\eta Q\preceq0\}.
\end{equation}
When the optimal solution $\eta^*$ of (\ref{eqdde}) exists, then by Theorem \ref{ellipthmcon3} we can claim that the Lorenz cone is an invariant set for the continuous system. 
Various celebrated SDO solvers, {e.g.}, SeDuMi, CVX, and
SDPT3, can be used to solve SDO problem (\ref{eqdde}).

\begin{corollary}\label{cor511}
If condition (\ref{ellip32con}) holds, then
\begin{equation}\label{con2772}
\max_{1\leq i\leq n-1}\left\{{u_i^T(A^T+A)u_i}\right\}\leq \eta\leq
u_n^T(A^T+A)u_n.
\end{equation}
\end{corollary}
\begin{proof}
The proof is similar to the one presented in the proof of  Corollary \ref{cor51}
by noting that $u_i^T(A^TQ+QA) u_i=2 (Au_i)^TQu_i$, and
$Qu_i=\lambda_i u_i$.
\end{proof}

\section{Examples}\label{sec:exam}
In this section, we present some simple examples to illustrate the invariance conditions presented in Section \ref{sec:invcond}. Since it is straightforward for discrete systems, we only present examples for continuous systems.
The following two examples consider polyhedral sets
for continuous systems.

\begin{example}
Consider the polyhedron $\mathcal{P}=\{(\xi,\eta)\,|\,\xi+\eta\leq1, -\xi+\eta\leq1,\xi-\eta\leq1,-\xi-\eta\leq1\}$, and the
continuous system $\dot{\xi}=-\xi,\dot{\eta}=-\eta.$
\end{example}

The solution of the system is $\xi(t)=\xi_0e^{-t}, \eta(t)=\eta_0e^{-t} $, so
$(\xi(t),\eta(t))\in \mathcal{P}$ for all $t\geq0,$ i.e., the polyhedron
is an invariant set for the continuous system provided that $(\xi_0,\eta_0)\in \mathcal{P}$. This can also be
verified by Theorem \ref{polythmcon}. We have
$$
H=-I_4, ~~G=\left[
          \begin{array}{rr}
            1 & 1 \\
            -1 & 1 \\
            1 & -1 \\
            -1 & -1 \\
          \end{array}
        \right],~~b=\left[
                      \begin{array}{c}
                        1 \\
                        1 \\
                        1 \\
                        1 \\
                      \end{array}
                    \right],
~~ A=-I_2,
$$
which satisfy  $HG=GA$ and $Hb\leq0$.  Thus  Theorem
\ref{polythmcon} yields that $\mathcal{P}$ is an
invariant set for this continuous system.

\begin{example}
Consider the polyhedral cone $\mathcal{C_P}$ generated by the
extreme rays $x^1=(1,1,1)^T,x^2=(-1,1,1)^T,x^3=(1,-1,1)^T,$ and
$x^4=(-1,-1,1)^T,$ and the continuous system
$\dot{\xi}=\xi,\dot{\eta}=\eta,\dot{\zeta}=\zeta.$
\end{example}

The solution of the system is
$\xi(t)=\xi_0e^t,~\eta(t)=\eta_0e^t,~\zeta(t)=\zeta_0e^t,$ thus one can easily verify
that the polyhedral cone is an invariant set for this continuous
system provided that $(\xi_0,\eta_0,\zeta_0)\in \mathcal{C_P}$. This can also be verified by Corollary \ref{coro:polycone}. We
have
$$
X=\left[
    \begin{array}{rrrr}
      1 & -1 & 1 & -1 \\
      1 & 1 & -1 & -1 \\
      1 & 1 & 1 & 1 \\
    \end{array}
  \right], ~\tilde{L}=I_4,~A=I_3,
$$
which satisfy  that $X\tilde{L}=AX$. Thus Corollary \ref{coro:polycone}
yields that $\mathcal{C_P}$ is an invariance
set for this continuous system.

The following two examples consider ellipsoids and Lorenz cones for continuous systems.
\begin{example}\label{exmp1}
Consider the  ellipsoid
$\mathcal{E}=\{(\xi,\eta)\;|\;\xi^2+\eta^2\leq1\}$, and the system
$\dot{\xi}=-\eta, \dot{\eta}= \xi.$
\end{example}

The solution of the system is $\xi(t)=\alpha\cos t+\beta\sin t$ and
$ \eta(t)=\alpha \sin t-\beta \cos t$, where $\alpha, \beta$ are two
parameters depending on the initial condition. The solution
trajectory is a circle, thus the system is invariant on this
ellipsoid. Also, we have
$$
A=\left[
    \begin{array}{cc}
      0 & -1 \\
      1 & 0 \\
    \end{array}
  \right], ~~
Q=I_2, ~~~A^TQ+QA= \left[
          \begin{array}{cc}
            0 & 0 \\
            0 & 0 \\
          \end{array}
        \right]\preceq0,
$$
which shows that, according to Theorem \ref{ellipthmcon1}, the
ellipsoid is an invariant set for this continuous system.

\begin{example}\label{exmp2}
Consider the Lorenz cone
$\mathcal{C_L}=\{(\xi,\eta,\zeta)\;|\;\xi^2+\eta^2\leq \zeta^2,
\zeta\geq0\}$, and the system
$\dot{\xi}=\xi-\eta, \dot{\eta}=\xi+\eta, \dot{\zeta}=\zeta.$
\end{example}

The solution is $\xi(t)=e^t(\alpha\cos t+\beta\sin t),$
$\eta(t)=e^t(\alpha\sin t-\beta\cos t)$ and $\zeta(t)=\gamma e^t,$
where $\alpha, \beta, \gamma$ are three parameters depending on the
initial condition. It is easy to verify that this Lorenz cone is an invariant set for the continuous
system.
Also, by letting $\eta\leq -2,$ we have
$$
A=\left[
    \begin{array}{ccc}
      1 & -1 & 0 \\
      1 & 1 & 0 \\
      0 & 0 & 1 \\
    \end{array}
  \right], ~~
Q=I_3, ~~A^TQ+QA+\eta Q= \left[
    \begin{array}{ccc}
      \eta+2 & 0 & 0 \\
      0 & \eta+2 & 0 \\
      0 & 0 & \eta+2 \\
    \end{array}
  \right]\preceq0,
$$
which shows that, according to Theorem \ref{ellipthmcon3}, the
Lorenz cone is an invariant set for this continuous system.


\section{Conclusions}\label{sec:con}

Invariant sets are important  both in the theory and for computational practice of dynamical systems. In this paper, we explore invariance conditions for four classic convex sets, for both linear discrete and continuous systems. In particular, these four convex sets are polyhedra, polyhedral cones, ellipsoids, and Lorenz cones, all of which have a wide range of applications in control theory.

In this paper, we present a novel, simple and unified method to derive  invariance conditions for linear dynamical systems. We first consider discrete systems, followed by continuous systems, since invariance conditions of the latter one are derived by using  invariance condition of the former one. For discrete systems, we introduce the Theorems of Alternatives, i.e., Farkas lemma and \emph{S}-lemma, to derive invariance conditions. We also show that by  applying the  \emph{S}-lemma one can extend  invariance conditions to any set represented by a quadratic inequality. The connection between discrete systems and continuous systems is built by using the forward or backward Euler methods, while the invariance is preserved with sufficiently small step size. Then we use elementary methods to derive invariance conditions for continuous systems. 
 This paper not only presents  invariance conditions of the four convex sets for continuous and discrete systems by using simple proofs, but also establishes a framework, which may be used for other convex sets as invariant sets, to derive  invariance conditions for both continuous and discrete systems.
 
Future research interests mainly focus on four directions. The first one is  extending the results of this  paper to nonlinear dynamical systems and more general sets. Some results on the extension to nonlinear system, one may refer to \cite{horv2016}. The second one is extending the study of invariant sets for discrete system on smooth manifold and on Lie groups, because discrete dynamical system has been 	extended to these more general settings \cite{fiori2, fiori1}. The third one is exploring the applications of the results in our paper in control and related fields. The fourth one is extending the invariance condition for partial differential equation and using other numerical methods, e.g., finite element method, finite difference method, adaptive time step methods, etc.


\section*{Acknowledgments}
This research is supported by a Start-up grant of Lehigh University and by
TAMOP-4.2.2.A-11/1KONV-2012-0012: Basic research for the development of
hybrid and electric vehicles. The TAMOP Project is supported by the European Union
and co-financed by the European Regional Development Fund. 

The authors are grateful to the referees for the constructive suggestions and comments.


\section*{References}
 \bibliographystyle{elsarticle-harv} 
 \bibliography{myref}







\end{document}